\documentclass[a4paper,11pt]{article}

\usepackage{fullpage}
\usepackage{amsbsy}
\usepackage{latexsym}
\usepackage{amsfonts}
\usepackage{amssymb, stmaryrd, mathrsfs}
\usepackage[usenames]{color}
\usepackage{amsmath,amsthm}
\usepackage{enumerate}

\usepackage{graphicx}

\usepackage[T1]{fontenc}

\usepackage{mathdots}

\def\RR{\rm \hbox{I\kern-.2em\hbox{R}}}
\def\NN{\rm \hbox{I\kern-.2em\hbox{N}}}
\def\ZZ{\rm {{\rm Z}\kern-.28em{\rm Z}}}
\def\CC{\rm \hbox{C\kern -.5em {\raise .32ex \hbox{$\scriptscriptstyle
|$}}\kern
-.22em{\raise .6ex \hbox{$\scriptscriptstyle |$}}\kern .4em}}
\def\vp{\varphi}
\def\<{\langle}
\def\>{\rangle}
\def\wt{\widetilde}
\def\i{\infty}

\def\nl{\newline}
\def\o{\overline}
\def\wt{\widetilde}

\def\cC{{\cal C}}

\def\cO{{\cal O}}
\def\cR{{\cal R}}

\def\R{\mathbb{R}}
\def\N{\mathbb{N}}

\def\E{\mathbb{E}}
\def\P{\mathbb{P}}

\def\Chi{\raise .3ex
\hbox{\large $\chi$}} \def\vp{\varphi}
\def\lsima{\hbox{\kern -.6em\raisebox{-1ex}{$~\stackrel{\textstyle<}{\sim}~$}}\kern -.4em}
\def\lsim{\hbox{\kern -.2em\raisebox{-1ex}{$~\stackrel{\textstyle<}{\sim}~$}}\kern -.2em}
\def\gsim{\hbox{\kern -.2em\raisebox{-1ex}{$~\stackrel{\textstyle>}{\sim}~$}}\kern -.2em}
\def\({\left (}
\def\){\right )}

\newcommand{\be}{\begin{equation}}
\newcommand{\ee}{\end{equation}}
\newcommand{\bea}{$$ \begin{array}{lll}}
\newcommand{\eea}{\end{array} $$}
\newcommand{\bi}{\begin{itemize}}
\newcommand{\ei}{\end{itemize}}
\newcommand{\iref}[1]{(\ref{#1})}

\newtheorem{theorem}{Theorem}[section]
\newtheorem{lemma}[theorem]{Lemma}
\newtheorem{prop}[theorem]{Proposition}
\newtheorem{cor}[theorem]{Corollary}

\theoremstyle{definition}
\newtheorem{definition}[theorem]{Definition}
\newtheorem{example}[theorem]{Example}

\newtheorem{remark}{Remark}[section]
\numberwithin{equation}{section}

\usepackage{algorithm}
\usepackage{algpseudocode}
\algrenewcommand\algorithmicrequire{\textbf{input:}}
\algrenewcommand\algorithmicensure{\textbf{output:}}

\newcommand{\eps}{\varepsilon}

\DeclareMathOperator{\sgn}{sgn}

\DeclareMathOperator{\Span}{Span}

\title{\Large{\textbf{Optimal sampling and Christoffel functions on general domains}}\thanks{}}
\author{Matthieu Dolbeault\thanks{Sorbonne Universit\'e, UPMC Univ Paris 06, CNRS, UMR 7598, Laboratoire Jacques-Louis Lions, 4 place Jussieu, 75005 Paris, France (dolbeault@ljll.math.upmc.fr)}, \ and Albert Cohen\thanks{Sorbonne Universit\'e, UPMC Univ Paris 06, CNRS, UMR 7598, Laboratoire Jacques-Louis Lions, 4 place Jussieu, 75005 Paris, France (cohen@ljll.math.upmc.fr)}}
\date{\today}

\begin{document}
\maketitle

\begin{abstract}
We consider the problem of reconstructing an unknown function $u\in L^2(D,\mu)$
from its evaluations at given sampling points $x^1,\dots,x^m\in D$, where $D\subset \R^d$ is a general
domain and $\mu$ a probability measure. The approximation is picked from a
linear space $V_n$ of interest where $n=\dim(V_n)$. Recent results \cite{CM1, DH,JNZ} have revealed that
certain weighted least-squares methods achieve near best (or instance optimal) approximation 
with a sampling budget $m$ that is proportional to $n$, up to a logarithmic factor $\ln(2n/\eps)$
where $\eps>0$ is a probability of failure. The sampling points should be picked at random
according to a well-chosen probability measure $\sigma$ 
whose density is given by the inverse Christoffel function that depends 
both on $V_n$ and $\mu$. While this approach is greatly facilitated when $D$ and $\mu$
have tensor product structure, it becomes problematic for domains $D$ with arbitrary geometry
since the optimal measure depends on an orthonormal basis of $V_n$ in $L^2(D,\mu)$ which is not explicitly given,
even for simple polynomial spaces. Therefore sampling according to this measure
is not practically feasible. One computational solution recently proposed in \cite{AH} relies 
on using the restrictions of an orthonormal basis of $V_n$ defined on a simpler bounding domain
and sampling according to the original probability measure $\mu$, in turn giving up on the optimal sampling budget $m\sim n$.
In this paper, we discuss practical sampling strategies, which amounts to 
using a perturbed measure $\wt \sigma$ that can be computed in an offline stage,
not involving the measurement of $u$, as recently proposed in \cite{Mig}. 
We show that near best approximation 
is attained by the resulting weighted least-squares method at near-optimal sampling budget and we discuss multilevel approaches that preserve optimality of the cumulated 
sampling budget when the spaces $V_n$ are iteratively enriched. 
These strategies rely on the knowledge of a-priori upper bounds $B(n)$ on the inverse Christoffel function for the space $V_n$ and the domain $D$. We establish bounds of the form $\cO(n^r)$ for spaces $V_n$ 
of multivariate algebraic polynomials, and for general 
domains $D$. The exact growth rate $r$ is established
depending on the regularity of the domain, in particular $r=2$ for domains with Lipschitz boundaries and $r=\frac {d+1}d$ for smooth domains. 

\medskip\noindent
\emph{MSC 2010:} 41A10, 41A65, 62E17, 65C50, 93E24
\end{abstract}

\section{Introduction}

\subsection{Reconstruction from point samples}

The process of reconstructing an unknown function 
$u$ defined on a domain $D\subset \R^d$ from its
sampled values 
\be
u^i=u(x^i)
\ee
at a set of point $x^1,\dots,x^m\in D$
is ubiquitous in data science and engineering. The sampled values 
may be affected by noise, making critical the stability properties 
of the reconstruction process. Let us mention three very different applicative 
instances for such reconstruction problems:
\begin{enumerate}
\item[(i)]
Statistical learning and regression: we observe $m$ independent realizations $(x^i,y^i)$ of a random
variable $z=(x,y)$ distributed according to an unknown measure, where $x\in D$ and $y\in\R$,
and we want to recover a function $x\mapsto v(x)$ that makes $|y-v(x)|$
 as small as possible in some given sense. If we use the quadratic loss $\E(|y-v(x)|^2)$, the minimizer
is given by the regression function 
\be
u(x)=\E(y|x).
\ee
and the observed $y^i$ may be thought of as the observation of $u(x^i)$ affected by noise.
\item[(ii)]
State estimation from measurements: the function $u$ represents the distribution of a physical
quantity (temperature, quantity of a contaminant, acoustic pressure) 
in a given spatial domain $D$ that one is allowed to measure by sensors placed
at $m$ locations $x^1,\dots,x^m$. These measurements can be affected by noise
reflecting the lack of accuracy of the sensors.
\item[(iii)]
Design of physical/computer experiments: $u$ is a quantity of interest that depends on the solution $f$ to a parametrized physical
problem. For example, $f=f(x)$ could be the solution to a PDE that depends on a vector $x=(x_1,\dots,x_d)\in D$ of $d$ physical parameters,
and $u$ could be the result of a linear form $\ell$ applied to $f$, that is, $u(x)=\ell(f(x))$.
We use a numerical solver for this PDE as a black box to evaluate $f$, and therefore $u$,
at $m$ chosen parameter vectors $x^1,\dots, x^m\in D$, and we now want 
to approximate $u$ on the whole domain $D$ from these computed values $u(x^i)$. 
Here, the discretization error of the solver may be considered as a noise affecting the true value.
\end{enumerate}

Contrarily to statistical learning, in the last two applications (ii) and (iii) the positions of the sample points $x^i$ are not realization of an 
unknown probability distribution. They can be selected by the user, which brings out the problem of 
choosing them in the best possible way. Indeed, measuring $u$ at the sample points may be costly: in (ii) 
we need a new sensor for each new point, and in (iii) a new physical experiment or run of a numerical solver.
Moreover, in certain applications, one may be interested in reconstructing many different instances of functions $u$.
Understanding how to sample in order to achieve the best possible trade-off between the sampling budget and the reconstruction performance
is one main motivation of this work. We first make our objective more precise
by introducing some benchmarks for the performance of the reconstruction process and sampling budget.

\subsection{Optimality benchmarks}
\label{optbench}

We are interested in controlling the distance 
\be
d(u,\wt u):=\|u-\wt u\|,
\ee
between $u$ and its reconstruction $\wt u=\wt u(u^1,\dots,u^m)$, measured in some given
norm $\|\cdot\|=\|\cdot\|_V$, where $V$ is a Banach function space that contains $u$.

For a given numerical method, the derivation of an error bound  
is always tied to some prior information on $u$. One most common way to 
express such a prior is in terms of membership of $u$ to a restricted class of functions,
for example a smoothness class. One alternate way is to express the prior in terms of 
approximability of $u$ by particular finite dimensional spaces. It is well-known that the
two priors are sometimes equivalent: many classical smoothness classes can be characterized in 
terms of approximability in some given norm by classical approximation spaces such as algebraic 
or trigonometric polynomials, splines or wavelets \cite{DL}.

In this paper, we adopt the second point of view, describing $u$ by its closeness
to a given subspace $V_n\subset V$ of dimension $n$: defining the best approximation error
\be
e_n(u):=\min_{v\in V_n} \|u-v\|,
\ee
our prior is that $e_n(u)\leq \eps_n$ for some $\eps_n>0$. One of our motivations is the rapidly expanding
field of {\it reduced order modeling} in which one searches for 
approximation spaces $V_n$ which are optimally designed to approximate families of solutions 
to parametrized PDEs. Such spaces differ significantly from the above-mentioned
classical examples. For example in the {\it reduced basis method}, 
they are generated by particular instances of 
solutions to the PDE for well chosen parameter values. We refer to \cite{CD}
for a survey on such reduced modeling techniques and their approximation capability.

In this context, one first natural objective is to build a reconstruction map
\be
(u^1,\dots,u^m) \mapsto \wt u_n \in V_n,
\ee
that performs almost as good as the best approximation error. We say that
a reconstruction map taking its value in $V_n$ is {\it instance optimal} with
constant $C_0\geq 1$ if and only if 
\be
\|u-\wt u_n\| \leq C_0\,e_n(u),
\label{instance}
\ee
for any $u\in V$. 

Obviously, instance optimality implies that if $u\in V_n$, the reconstruction map should
return an exact reconstruction $\wt u_n=u$. For this reason, instance optimality can only be hoped for if
the sampling budget $m$ exceeds the dimension $n$. This leads us to introduce a second notion 
of optimality: we say that the 
sample is {\it budget optimal} with constant $C_1\geq 1$ if
\be
m\leq C_1\,n.
\label{budget}
\ee

Let us stress that in many relevant settings, we do not work with a single space $V_n$
but a sequence of nested spaces 
\be
V_1\subset V_2 \subset \dots\subset V_n\subset \dots
\ee
so that $e_n(u)$ decreases as $n$ grows. Such a hierarchy could either be fixed in advance
(for example when using polynomials of degree $n$), or adaptively chosen as we collect more samples
(for example when using locally refined piecewise polynomials or finite element spaces).
Ideally, we may wish that the constants $C_0$ and $C_1$ are independent of $n$.
As it will be seen, a more accessible goal is that only one of the two constants is independent of $n$, while the other grows at most logarithmically in $n$.

Another way of relaxing instance optimality is to request the weaker property of {\it rate optimality},
which requires that for any $s>0$ and $u\in V$,
\be
\sup_{n\geq 1}\,n^s \,\|u-\wt u_n\|_V \leq C \,\sup_{n\geq 1}\,n^s\, e_n(u),
\ee
where $C\geq 1$ is a fixed constant. In other words, the approximant produced by the reconstruction method should converge
at the same polynomial rate as the best approximation. 

In the context where the spaces $V_n$ are successively refined, even if the reconstruction method is instance and budget optimal
for each value of $n$, the cumulated sampling budget until the $n$-th refinement step is in principle
of the order
\be
m(n)\sim 1+2+\dots+n \sim n^2,
\ee
if samples are picked independently at each step. A natural question is whether the samples used until stage $k$ can be,
at least partially, recycled for the computation of $\wt u_{k+1}$, in such a way that the cumulated sampling budget
$m(n)$ remains of the optimal order $\cO(n)$. This property will be ensured for example if for each $n$, the samples
are picked at points $\{x^1,\dots,x^{m(n)}\}$ that are the sections of a unique infinite sequence $\{x^m\}_{m\geq 1}$,
with $m(n)\sim n$, which means that all previous samples are recycled. We refer to this property as {\it hierarchical} sampling. It is also referred to as {\it online machine learning}
in the particular above-mentioned applicative context (i). 

\subsection{Objectives and layout}

The design of sampling and reconstruction strategies that combine budget 
and instance (or rate) optimality, together with the above progressivity prescription, turns out
to be a difficult task, even for very classical approximation spaces $V_n$ such as polynomials.

In the next section \S 2, we illustrate this difficulty by first discussing the example of reconstruction 
by {\it interpolation} for which the sampling budget is optimal but instance optimality with error measured
in the $L^\infty$ norm generally fails by a large amount. We then recall recent results
\cite{ABC,CM1,DH,JNZ} revealing that one can get much closer to these optimality objectives
by {\it weighted least-squares} reconstruction methods. In this case, we estimate the approximation error
in $V=L^2(D,\mu)$ where $\mu$ is an arbitrary but fixed probability measure.
The sampling points are picked at random
according to a different probability measure $\sigma^*$ that depends on $V_n$ and $\mu$:
\be
d\sigma^*(x)=\frac {k_n(x)}{n}\,d\mu(x).
\ee
Here $k_n$ is the {\it inverse Christoffel function} defined by 
\be
k_n(x)=\sum_{j=1}^n |L_j(x)|^2,
\label{invchris}
\ee
where $(L_1,\dots,L_n)$ is any $L^2(D,\mu)$-orthonormal basis of $V_n$. By Cauchy-Schwarz inequality, it is readily seen that this function is characterized by the
extremality property
\be
k_n(x)=\max_{v\in V_n}\frac{|v(x)|^2}{\|v\|^2},
\label{knext}
\ee
where $\|v\|:=\|v\|_V=\|v\|_{L^2(D,\mu)}$.
Then, instance optimality is achieved in a probabilistic sense
with a sampling budget $m$ that is proportional to $n$, up to a logarithmic factor $\ln(2n/\eps)$
where $\eps>0$ is a probability of failure which comes as an additional term in the
instance optimality estimate
\be
\E(\|u-\wt u_n\|^2)\leq C_0\,e_n(u)^2+\cO(\eps ).
\ee
 It is important to notice that $\sigma^*$ differs from $\mu$
and that the standard least-squares method using a sample drawn according to $\mu$ is
generally {\it not} budget optimal in the sense that instance optimality requires $m$  
to be larger than 
the quantity
\be
K_n:=\|k_n\|_{L^{\infty}}=\sup_{x\in D} |k_n(x)|=\max_{v\in V_n}\frac{\|v\|_{L^{\infty}}^2}{\|v\|_{L^2}^2},
\ee
which may be much larger than $n$, for instance $\cO(n^2)$ or worse, see \cite{CDL}
as well as \S 5.

While these results are in principle attractive since they apply to arbitrary spaces $V_n$,
measures $\mu$ and domains $D$, the proposed sampling strategy is highly facilitated
when $D$ is a tensor-product domain and $\mu$ is the tensor-product of a simple univariate
measure, so that an $L^2(D,\mu)$-orthonormal basis of $V_n$ can be explicitly provided. 
This is the case for example when using multivariate algebraic or trigonometric polynomial
spaces with $\mu$ being the uniform probability measure on $[-1,1]^d$ or $[-\pi,\pi]^d$.
For a general domain $D$ with arbitrary - possibly irregular - geometry, 
the orthonormal basis cannot be explicitly computed, even for simple polynomial spaces. 
Therefore sampling according to the optimal measure $\sigma^*$ is not feasible.

Non-tensor product domains $D$ come out naturally in all the above mentioned 
applicative settings (i)-(ii)-(iii). For example, in design of physical/computer experiment, 
this reflects the fact that while the individual parameters $x_j$ could range in intervals $I_j$ for $j=1,\dots,d$,
not all values $x$ in the rectangle $R=I_1\times \dots\times I_d$ are physically admissible.
Therefore, the function $u$ is only accessible and searched for
in a limited domain $D\subset R$.

One practical solution proposed in \cite{AH} consists in sampling according to 
the measure $\mu$
and solving the least-squares problem using the restriction of an orthonormal basis of $V_n$ 
defined on a simpler tensor product bounding domain, which generally gives rise to a frame. This approach is feasible for example
when $\mu$ is the uniform probability measure and when the inclusion of a point in $D$ can be 
numerically tested. Due to the use of restricted bases, the resulting Gramian matrix
which appears in the normal equations is ill-conditioned or even singular, 
which is fixed by applying a pseudo-inverse after thresholding the smallest singular values
at some prescribed level. Budget optimality is generally lost in this approach
since one uses $\mu$ as a sampling measure.

In this paper, we also work under the assumption that we are able to sample according to $\mu$,
but we take a different path, which is exposed in \S 3. In an {\it offline} stage, 
we compute an approximation $\wt k_n$ to the inverse Christoffel function,
which leads to a measure $\wt \sigma$ that may
be thought as a perturbation of the optimal measure $\sigma^*$. 
We may then use $\wt \sigma$ to define the sampling points 
$\{x^1,\dots,x^m\}$ 
and weights. In the {\it online} stage, we 
perform the weighted least-squares reconstruction strategy
based on the measurement of $u$ at these points. 
Our first result is that if $\wt k_n$ is equivalent to $k_n$,
we recover the stability and instance optimality results from \cite{CM1}
at near-optimal sampling budget $m\sim n\ln(2n/\eps)$.

One approach for computing $\wt k_n$, recently proposed
in \cite{Mig}, consists in drawing a first sample $\{z^1,\dots,z^M\}$ according to $\mu$
and defining $\wt k_n$ as the inverse Christoffel function with respect to the discrete measure 
associated to these points. In order to ensure an equivalence between $k_n$ and $\wt k_n$ with high probability, 
the value of $M$ needs to be chosen larger than $K_n$ which is unknown
to us. This can be ensured by asking that $M$ is larger than a known
upper bound $B(n)$ for $K_n$. The derivation of such bounds
for general domains is one of the objectives of this paper. 
We also propose an empirical strategy 
for choosing $M$ that does not require the knowledge of 
an upper bound and appears to be effective in our numerical tests.
In all cases, the size $M$ of the offline sample 
could be of order substantially larger than $\cO(n)$. 
However, this first set of points is only used
in the offline stage to perform computations that produce the perturbed measure $\wt \sigma$, and
 {\it not} to evaluate the function $u$ which, as previously explained, is the costly aspect in the 
targeted applications and could also occur for many instances of $u$. These more costly evaluations of $u$
only take place in the online stage  
at the $x^i$, therefore at near-optimal sampling budget. 

In the case where $K_n$, or its available bound $B(n)$, grows very fast with $n$, the complexity of the offline stage in this approach becomes itself prohibitive. In order to mitigate this defect, we introduce in \S 4 a multilevel approach where the approximation $\wt k_n$ of $k_n$ is produced by successive space refinements 
\be
V_{n_1}\subset \dots \subset V_{n_q}, \quad n_q=n,
\ee
which leads to substantial computational savings under mild assumptions. This setting also allows us to produce nested sequences of evaluation points $\{x^1,\dots,x^{m_p}\}$ where $m_p$ grows similar to $n_p$ up to a logarithmic factor, therefore complying with the previously
invoked prescription of hierarchical sampling.

In \S 5 we turn to the study of the inverse Christoffel function $k_n$
in the case of algebraic polynomial spaces on general multivariate domains $D\subset \R^d$. We establish pointwise and global upper and lower bounds for $k_n$ that depend
on the smoothness of the boundary of $D$.
We follow an approach adopted in \cite{Pry} for a particular class of domains with piecewise smooth boundary, namely
comparing $D$ with simpler reference domains for which the inverse Christoffel function can be estimated.
We obtain bounds with growth rate $\cO(n^r)$ where the value $r=2$ for Lipschitz domains and $r=\frac {d+1}d$ for
smooth domains is proved to be sharp. We finally give a systematic approach that also describes the sharp growth rate
for domains with cusp singularities.

We close the paper in \S 6 with various numerical experiments that confirm
our theoretical investigations. In the particular case of multivariate algebraic polynomials,
the sampling points tend to concentrate near to the exiting corner or cusp singularities
of the domain, while they do not at the reintrant singularities, as predicted 
by the previous analysis of the inverse Christoffel function. 

\section{Meeting the optimality benchmarks}

\subsection{Interpolation}

One most commonly used strategy to reconstruct functions from point values
is interpolation. Here we work in the space $V=\cC(D)$ of continuous and bounded
functions equipped with the $L^\infty$ norm. For the given space $V_n$, and
$n$ distinct points $x^1,\dots,x^n\in D$ picked in such way that the map $v\mapsto (v(x^1),\dots,v(x^n))$ is 
an isomorphism from $V_n$ to $\R^n$, we define the corresponding interpolation 
operator $\mathcal I_n: \cC(D)\to V_n$ by the interpolation condition
\be
\mathcal I_n u(x^i)=u(x^i), \quad i=1,\dots,n.
\ee
The interpolation operator is also expressed as
\be
\mathcal I_n u=\sum_{i=1}^n u(x^i)\,\ell_i,
\ee
where $\{\ell_1,\dots,\ell_n\}$ is the Lagrange basis of $V_n$ defined by the conditions $\ell_i(x^j)=\delta_{i,j}$.
Interpolation is obviously budget optimal since it uses $m=n$ points, that is, $C_1=1$ in \iref{budget}.
On the other hand, it does not guarantee instance optimality: the constant $C_0$ in \iref{instance} is
governed by the {\it Lebesgue constant}
\be
\Lambda_n=\| \mathcal I_n\|_{L^\infty\to L^\infty}=\max_{x\in D}\, \sum_{i=1}^n |\ell_i(x)|.
\ee
Indeed, since $\|u-\mathcal I_n u\|_{L^\infty}\leq \|u-v\|_{L^\infty}+\|\mathcal I_nu-\mathcal I_nv\|_{L^\infty}$ for any $v\in V_n$, one has
\be
\|u-\mathcal I_n u\|_{L^\infty} \leq (1+\Lambda_n) e_n(u).
\ee
The choice of the points $x^i$ is critical to control the growth of $\Lambda_n$ with $n$. For example in
the elementary case of univariate algebraic polynomials where $D=[-1,1]$ and $V_n=\P_{n-1}$, it is well known
that uniformly spaced $x^i$ result in $\Lambda_n$ growing exponentially, at least like $2^n$, while the slow growth $\Lambda_n\sim \ln(n)$
is ensured when using the Chebychev points $x^i=\cos\(\frac{2i-1}{2n} \pi\)$
for $i=1,\dots n$.
Unfortunately, there is no general guideline to ensure such a slow growth for more general hierarchies of spaces $(V_n)_{n\geq 1}$
defined on multivariate domains $D\subset \R^d$. As an example, in the simple case of 
the bivariate algebraic polynomials $V_n=\P_p$ where $n=\frac{(p+1)(p+2)}2$ and a general polygonal
domain $D$, a choice of points that would ensure a logarithmic growth of the Lebesgue constant is 
an open problem.

There exists a general point selection strategy that ensures linear behaviour of the Lebesgue constant
for any space $V_n$ spanned by $n$ functions $\{\vp_1,\dots,\vp_n\}$: it consists in choosing $(x^1,\dots,x^n)$
which maximizes over $D^n$ the determinant of the collocation matrix
\be
M(x^1,\dots,x^n)=(\phi_i(x^j))_{i,j=1,\dots,n},
\ee
Since the $j$-th element of the Lagrange basis is given by
\be
\ell_j(x)= \frac{{\rm det}(M(x^1,\dots,x^{j-1},x,x^{j+1},\dots,x^n))}{{\rm det}(M(x^1,\dots,x^n))},
\ee
the maximizing property gives that $\|\ell_j\|_{L^\infty}\leq 1$ and therefore $\Lambda_n \leq n$. In the particular
case of the univariate polynomials where $D=[-1,1]$ and $V_n=\P_{n-1}$, this choice corresponds to the
Fekete points, which maximize  the product $\prod_{i\neq j} (x^i-x^j)$.

While the above strategy guarantees the $\cO(n)$ behaviour of $\Lambda_n$, its
main defect is that it is computationally unfeasible if $n$ or $d$ is large, 
since it requires solving a non-convex optimization problem in dimension $d\,n$.
In addition to this, for a given hierarchy of spaces $(V_n)_{n\geq 1}$, the 
sampling points $S_n=\{x^1,\dots,x^n\}$ generated
by this strategy do not satisfy the nestedness property $S_n\subset S_{n+1}$. 

A natural alternate strategy that ensures nestedness consists in selecting the points by a
stepwise greedy optimization process: given $S_{n-1}$, define the next point $x^{n}$
by maximizing over $D$ the function $x\mapsto {\rm det}(M(x^1,\dots,x^{n-1},x))$.
This approach was proposed in \cite{Magic} in the context of reduced basis approximation
and termed as {\it magic points}. It amounts to solving at each step a non-convex optimization problem in the
more moderate dimension $d$, independent of $n$. However there exists no general bound on $\Lambda_n$
other than exponential in $n$. In the univariate polynomial case, this strategy yields the 
so-called {\it Leja points} for which it is only known that the Lebesgue constant grows sub-exponentially
although numerical investigation indicates that it could behave linearly. In this very 
simple setting, the bound $\Lambda_n\leq n^2$ could be established in \cite{Chkifa},
however using a variant where the points are obtained by projections of the complex 
Leja points from the unit circle to the interval $[-1,1]$.

In summary, while interpolation uses the optimal sampling budget $m=n$, it fails by a large amount
in achieving instance optimality, especially when asking in addition for the nestedness 
of the sampling points, even for simple polynomial spaces.

\subsection{Weighted least-squares}

In order to improve the instance optimality bound, we allow ourselves to collect more data on the function $u$ by increasing the number $m$ of sample points, compared to the critical case $m=n$ studied before, and construct
an approximation $\wt u_n$ by a least-squares fitting procedure. This relaxation of the problem gives more flexibility on the choice of the sample points: for instance, placing two of them too close will only waste one evaluation of $u$, whereas this situation would have caused ill-conditioning and high values of $\Lambda_n$ in interpolation. 
It also leads to more favorable results in terms of instance optimality as we next recall.

Here, and in the rest of this paper, we assess the error in the $L^2$ norm 
\be
\|v\|=\|v\|_{L^2(D,\mu)},
\ee
where $\mu$ is a fixed probability measure, which can be arbitrarily chosen by the user depending
on the targeted application. For example, if the error has the same significance at all points of $D$, one
is naturally led to use the uniform probability measure 
\be
d\mu:=|D|^{-1} dx.
\ee
In other applications such as uncertainty quantification
where the $x$ variable represents random parameters that 
follow a more general probability law $\mu$, the use of this 
specific measure is relevant since the 
reconstruction error may then be interpreted as the mean-square risk
\be
\|u-\wt u_n\|_{L^2(D,\mu)}^2=\E_x(|u(x)-\wt u_n(x)|^2).
\ee
Once the evaluations of $u(x^i)$ are performed, the {\it weighted-least squares methods}
defines $\wt u_n$ as the solution of the minimization problem
\be
\min_{v\in V_n} \sum_{i=1}^m w(x^i) \,|u(x^i)-v(x^i)|^2,
\label{wls}
\ee
where $w(x^1),\dots,w(x^m)>0$ are position-dependent weights. 
The solution to this problem is unique under the assumption that 
no function of $V_n\setminus \{0\}$ vanishes at all the $x^i$. 
Notice that in the limit $m=n$, the minimum in \iref{wls} is zero
and attained by the interpolant at the points $x^1,\dots,x^n$, which as previously discussed suffers from a severe lack of instance optimality. 

The results from \cite{CM1} provide with a general strategy to select the points $x^i$ 
and the weight function $w$ in order to reach instance and budget optimality,
in a sense that we shall make precise. In this approach, the points $x^i$ are drawn at random
according to a probability measure $\sigma$ on $D$, that generally differs from $\mu$,
but with respect to which $\mu$ is absolutely continuous.
One then takes for $w$ the corresponding Radon-Nikodym derivative, so that
\be
w(x)\,d\sigma(x)=d\mu(x).
\label{compatible}
\ee
This compatibility condition ensures that we recover a minimization in 
the continuous norm $\|\cdot\|$ as $m$ tends to infinity:
\be
\frac{1}{m}\sum_{i=1}^m  w(x^i) \,|u(x^i)-v(x^i)|^2\underset{m\rightarrow \infty}{\overset{a.s.}{\longrightarrow}}\int_D w\, |u-v|^2\,d \sigma=\int_D|u-v|^2\,d\mu=\|u-v\|^2.
\ee
Here we may work under the sole assumption that $u$ belongs to the space $V=L^2(D,\mu)$, 
since pointwise evaluations of $u$ and $w$ will be almost surely well-defined. 
In return, since $\wt u_n$ is now stochastic, the $L^2$ estimation error will only be 
assessed in a probabilistic sense, for example by considering the mean-square error,
\be
\E(\|u-\wt u_n\|^2)=\E_{\otimes^m \sigma}(\|u-\wt u_n\|^2)
\ee
The weighted least-square approximation may be viewed as the
orthogonal projection $\wt u_n=P^m_n u$ onto $V_n$ for the discrete $\ell^2$ norm
\be
\|v\|_m^2:=\frac{1}{m}\sum_{i=1}^m w(x^i)\,|v(x^i)|^2,
\label{discrete_norm}
\ee
in the same way that the optimal approximation
\be
u_n:=\underset{v\in V_n}{\arg\min}\,\|u-v\|=P_n u
\ee
is the orthogonal projection for the continuous $L^2(D,\mu)$ norm. 
A helpful object for comparing these two norms on $V_n$ is the Gramian matrix
\be
G:=(\<L_j,L_k\>_m)_{j,k=1,\dots,n},
\label{gramian}
\ee
where $(L_1,\dots,L_n)$ is any $L^2(D,\mu)$-orthonormal basis of $V_n$. Indeed, for all $\delta>0$,
\be
\|G-I\|_{2}\leq \delta\Longleftrightarrow  \ (1-\delta)\|v\|^2\leq \|v\|_m^2\leq (1+\delta)\|v\|^2,\quad v \in V_n,
\label{normequivgen}
\ee
where $\|M\|_2$ denotes the spectral norm of an $n\times n$ matrix $M$.
As noted in $\cite{CDL}$ in the case of standard least-squares, and in \cite{CM1} for the weighted case,
$G$ can be seen as a mean of $m$ independent and identically distributed matrices 
\be
X^i:=(w(x^i)\,L_j(x^i)\,L_k(x^i))_{j,k=1,\dots,n}
\ee
satisfying $\E(X^i)=I$, so $G$ concentrates towards the identity as $m$ grows to infinity. This concentration can be estimated by a matrix Chernoff bound, such as Theorem 1.1 in the survey paper \cite{Tr}. As observed in \cite{CM1}, for the particular value $\delta=\frac 1 2$, this inequality rewrites as follows,
in our case of interest. 

\begin{lemma}
\label{tropp}
For any $\eps>0$, under the sampling budget condition
\be
m \geq \gamma\, \| w\, k_n\|_{L^\infty}\,\ln (2n/\eps),
\label{conditionCM}
\ee
where $\gamma:=\(3/2\,\ln (3/2)-1/2\)^{-1} \approx 9.242$, one has
 $\Pr(\|G-I\|_{2}\leq 1/2)\geq 1-\eps$.
\end{lemma}

An estimate comparing for the estimator $\|u-\wt u_n\|$ with $e_n(u)$ can be obtained
when imposing that $\|G-I\|_{2}\leq 1/2$, as expressed in the following 
which is proved in \cite{CM1}.

\begin{lemma}
\label{lemcomputation}
One has
\be
\E\(\|u-\wt u_n\|^2 \Chi_{\|G-I\|_{2}\leq 1/2}\) \leq \(1+\frac{4}{m}\|w\,k_n\|_{L^\infty}\)\,e_n(u)^2.
\label{condest}
\ee
\end{lemma}

On the other hand, the estimator $\wt u_n$ obtained by solving \iref{wls} is not reliable 
in the event where $G$ becomes singular, which leads to modify its definition in various
ways: 
\begin{enumerate} 
\item
If one is able to compute $\|G-I\|_2$, one may
condition the estimator to the event $\|G-I\|_2\leq\frac{1}{2}$ by defining
\be
\wt u_n^C:=\wt u_n\,\Chi_{\|G-I\|_2\leq 1/2},
\label{uC}
\ee
that is, we take $\wt u_n^C=0$ if $\|G-I\|_2>\frac{1}{2}$. 
\item
If a uniform bound $\|u\|_{L^\infty(D)}\leq \tau$ is known, one may introduce a truncated estimator 
 \be
 \wt u_n^T:=T_\tau\circ \wt u_n,
 \label{uT}
 \ee
where $T_\tau(y):=\min\{\tau,|y|\}\sgn(y)$.
\end{enumerate}

The main results from \cite{CM1}, that we slightly reformulate below, show that these estimators are instance optimal 
in a probabilistic sense. Throughout the rest of the paper, $\gamma$ denotes the same constant as in Lemma \ref{tropp}.

\begin{theorem}
\label{cohenmigliorati}
Under the sampling budget condition 
\be
m \geq \gamma\, \| w \,k_n\|_{L^\infty}\,\ln (2n/\eps),
\label{budgetbound}
\ee
the weighted least-squares estimator satisfies 
\be
\E\,(\|u-\wt u_n\|^2 \Chi_{\|G-I\|_{2}\leq 1/2})\leq \(1+\eta(m)\)\,e_n(u)^2.
\label{boundwls}
\ee
The conditionned and truncated estimators satisfy the convergence bounds
\be
\E\,(\|u-\wt u_n^C\|^2)\leq \(1+\eta(m)\)\,e_n(u)^2+\|u\|^2\,\eps,
\label{bounduC}
\ee
and
\be
\E\,(\|u-\wt u_n^T\|^2)\leq \(1+\eta(m)\)\,e_n(u)^2+4\,\tau^2\,\eps,
\label{bounduT}
\ee
where $\eta(m)=\frac{4}{m}\|w\,k_n\|_{L^\infty}\leq \frac 4 {\gamma\ln (2n/\eps)} \to 0$, as $n\to \infty$ or $\eps\to 0$.
\end{theorem}

\begin{proof} 
The bound \iref{boundwls} follows directly from Lemma \ref{lemcomputation} and the assumption on $m$.
In the event $\|G-I\|_2>\frac{1}{2}$, of probability less than $\eps$ by Lemma \ref{tropp}, one can use the bounds \be
\|u-\wt u_n^C\|^2=\|u\|^2\quad\text{and}\quad\|u-\wt u_n^T\|^2\leq 4\tau^2.
\ee
 Otherwise, one has
 \be\|u-\wt u_n^C\|^2\leq \|u-\wt u_n\|^2\quad\text{and}\quad\|u-\wt u_n^T\|^2\leq \|u-\wt u_n\|^2.
 \ee
This leads to \iref{bounduC} and \iref{bounduT}.
\end{proof}

\begin{remark}
The above result shows that the estimators $\wt u_n^C$ and 
$\wt u_n^T$ achieve instance optimality in expectation up to 
additional error terms of order $\cO(\eps)$, accounting
for the event $\{\|G-I\|_{2}>1/2\}$. Note that $\eps$ only
influences the constraint on the sampling budget logarithmically. In particular, if $e_n(u)$
decreases like $n^{-r}$ for some $r>0$, these estimators are rate optimal
by taking $\eps$ less than $n^{-2r}$, which thus affects the constraint on sampling budget
by a factor $\cO(\ln(n))$. 
\end{remark}

\begin{remark} 
\label{remredraw}
One way to achieve instance optimality in expectation without an additional error term 
consists in redrawing the points $\{x^1,\dots,x^m\}$ until one observes that $\|G-I\|_2\leq \frac 1 2$,
as proposed in \cite{HNP}.
Denoting by $u_n^*$ the weighted least-square estimator corresponding to this draw, we
find that under the sampling budget \iref{budgetbound}, one has
\be
\E(\|u-u_n^*\|^2)=\E\(\|u-\wt u_n\|^2\; \Big | \; \|G-I\|_2\leq \frac 1 2\)
\leq \frac 1{1-\eps} \E\(\|u-\wt u_n\|^2\Chi_{\|G-I\|_2\leq \frac 1 2}\),
\ee
and thus
\be
\E(\|u-u_n^*\|^2)\leq \frac 1{1-\eps}(1+\eta(m)) e_n(u)^2.
\label{instancest}
\ee
The sampling budget condition also ensures a probabilistic control on the number 
of required redraws.
\end{remark}

Now the natural objective is to find a weight function $w$ that makes $\|w \,k_n\|_{L^\infty}$ small in order to minimize the sampling budget. Since
\be
\|w \,k_n\|_{L^\infty}\geq \int_D w \,k_n\, d\sigma=\int_D k_n\, d\mu=n,
\ee
with equality attained for the weight function
\be
w^*:=\frac{n}{k_n}=\frac{n}{\sum_{j=1}^n|L_j|^2},
\ee
this theorem shows that the choice of sampling measure
\be
d\sigma^* =\frac{1}{w^*}\,d\mu=\frac{k_n}{n}\,d\mu
\ee
 is optimal, in the sense that the above instance optimality results are achieved
 with a near-optimal sampling budget $m\sim n$ up to logarithmic
 factors.
 
As already explained in the introduction, when working on a general domain $D$,
we face the difficulty that the orthonormal basis $(L_1,\dots, L_n)$
cannot be exactly computed, and therefore the optimal $w^*$ and $\sigma^*$
are out of reach. The next section proposes computable alternatives
$\wt w$ and $\wt \sigma$ that still yield similar instance optimality results at
near-optimal sampling budget.

\section{Near-optimal sampling strategies on general domains}

\subsection{Two steps sampling strategies}

The sampling and reconstruction strategies that we discuss proceed in two steps:
\begin{enumerate}
\item In an offline stage we search for an approximation 
to the Christoffel function 
$k_n$. For this purpose, we sample $z^1,\dots,z^M\in D$ according to $\mu$,
use these sampling points to compute an orthonormal basis $(\wt L_1,\dots,\wt L_n)$ with respect to the induced discrete inner product. The approximation to the Christoffel function 
is then $\wt k_n=\sum_{j=1}^n |\wt L_j|^2$. As we explain further, one objective is
to guarantee that $\wt k_n$ and $k_n$ are pointwise equivalent.
We define the sampling measure $\wt \sigma$
as proportional to $\wt k_n \,\mu$  and draw the points $x^1,\dots,x^m$ 
according to this measure.
\item In an online stage, we evaluate $u$ at the sampling points 
$x^i$ and construct an estimate $\wt u_n$ by the weighted least-squares method.
\end{enumerate}

In the offline stage $M$ could be much larger than $n$, however it should be
understood that the function $u$ is only evaluated in the online stage at the $m$ point $x^i$ which will be seen to have optimal cardinality $m\sim n$ up to logarithmic factors.

The two main requirements in these approaches are the data of a
(non-orthogonal) basis $(\phi_1,\dots,\phi_n)$ of $V_n$ and the ability to sample
according to the measure $\mu$. When $D\subset \R^d$ is a general 
multivariate domain, one typical setting for this second assumption 
to be valid is the following:
\begin{itemize}
\item
There is a set $R$ containing $D$ such that 
$\mu$ is the restriction of a measure $\mu_R$ 
which can easily be sampled.
\item
Membership of a point $x$ to the set $D$ can be
efficiently tested, that is, $\Chi_D$ is easily computed.
\end{itemize}
This includes for instance the uniform probability measure on 
domains described by general systems of algebraic inequalities (such
as polyhedrons, ellipsoids..), by including such domains $D$
in a rectangle $R=I_1\times \dots \times I_d$ on which sampling according to 
the uniform measure can be done componentwise. Then
the $z^i$ are produced by sampling according to $\mu_R$
and rejecting the samples that do not belong to $D$.
The offline stage is described more precisely as follows.
\nl
\nl
{\bf Algorithm 1.} Draw a certain  number $M$ of points  $z^1, \dots, z^{M}$ independently according to $\mu$, and construct from $(\phi_j)_{j=1,\dots,n}$ an
orthonormal basis $(\wt L_j)_{j=1,\dots,n}$ of $V_{n}$
with respect to the inner product
\be
\<u,v\>_M:=\frac{1}{M}\,\sum_{i=1}^{M}\,u(z^i)\,v(z^i).
\label{Minner}
\ee
Then define
\be
\wt k_n(x)=\sum_{j=1}^n |\wt L_j(x)|^2,
\ee
the approximate inverse Christoffel function, and the corresponding sampling measure
\be
d\wt \sigma:=\alpha\,\frac {\wt k_n}n \,d\mu,
\ee
where $\alpha$ is the normalization factor such that $\alpha \int_D \wt k_n \,d\mu=n$. 
\nl

Note that the factor $\alpha$ is unknown to us but its value is not needed in typical sampling strategies, such as
rejection sampling or MCMC. In contrast to $k_n$, the function $\wt k_n$ is stochastic since it depends on the
drawing of the $z^i$.  In the online stage, we sample $x^1,\dots,x^m$ independently according to $d\wt \sigma$.
We then measure $u$ at the points $x^i$, and define
the estimator $\wt u_n\in V_n$ as the solution to 
the weighted least-squares problem 
\be
\min_{v\in V_n} \sum_{i=1}^m \wt w(x^i)\,|u(x^i)-v(x^i)|^2, 
\label{wls2}
\ee
with $\wt w= \frac{n}{\alpha \wt k_n}$.
This least-squares problem can be solved explicitly by computing $\wt u_n=P^m_nu$ as the orthogonal projection of $u$ on $V_n$ with respect to the inner product from \iref{discrete_norm}
\be
\<u,v\>_m:=\frac 1 m\sum_{i=1}^m \wt w(x^i)\,u(x^i)\,v(x^i).
\label{minner}
\ee

\begin{remark}
There are now two levels of stochasticity: the draw of the $z^i$ and the subsequent draw of the $x^i$.
We sometimes use the symbols $\E_z$ and $\Pr_z$ referring to the first draw, and $\E_x$ and $\Pr_x$ 
referring to the second draw given the first one, while $\E$ and $\Pr$ refer to both draws.
\end{remark}

We keep the notations $G$ and $\wt u_n^T$ from \iref{gramian} and \iref{uT}.
In the following section, we establish instance optimal convergence results under near optimal sample
complexity $m$ similar to \iref{boundwls} and \iref{bounduT} in Theorem \ref{cohenmigliorati}.
On the other hand we do not consider the conditioned estimator $\wt u_n^C$ any further since we do not have access to the 
matrix $G$ which would require the knowledge of the functions $L_j$. The derivation of a computable
estimator that satisfies a similar estimate as $\wt u_n^C$ is an open question.
We also discuss the required sample complexity $M$ of the offline stage. 

\subsection{Convergence bounds and sample complexity}

Our principle objective is to ensure the uniform framing
\be
c_1 \,k_n(x) \leq \wt k_n(x)\leq  c_2 \,k_n(x),  \quad x\in D,
\label{compk}
\ee
for some known constants $0<c_1\leq c_2$. Our motivation 
is that instance optimal convergence bounds with near-optimal 
sampling budget hold under this framing, as expressed by the following result.

\begin{theorem}
\label{theoframe}
Assume that \iref{compk} holds for some $0<c_1\leq c_2$ and let $c=\frac {c_2}{c_1}\geq 1$. Then, under the sampling budget condition 
\be
m\geq c\,\gamma \,n\,\ln(2n/\eps),
\label{coptsamplewithc}
\ee
one has ${\Pr}_x \(\|G-I\|_2\geq \frac 1 2\) \leq \eps$. In addition, one has the convergence bounds
\be
\E_x\(\|u-\wt u_n\|^2 \Chi_{\|G-I\|_2\leq \frac 1 2}\)\leq (1+\eta(m)) \,e_n(u)^2,
\label{bound33}
\ee
and
\be
\E_x(\|u-\wt u_n^T\|^2)\leq (1+\eta(m))\, e_n(u)^2+ 4\,\eps \,\tau^2,
\label{bound3}
\ee
where $\eta(m)=4\, c\,\frac{n}{m}\leq\frac {4}{\gamma\ln(2n/\eps)}$.
\end{theorem}

\noindent
\begin{proof}
It is an immediate application of the results from \S 2.2.
Indeed
\be
\|\wt w \,k_n\|_{L^\infty}=\left\|\frac{n\, k_n}{\alpha \,\wt k_n}\right\|_{L^\infty}=\left\|\frac{k_n}{\wt k_n}\right\|_{L^\infty}\int_D \wt k_n\,d\mu\leq \left\|\frac{k_n}{\wt k_n}\right\|_{L^\infty}\bigg\|\frac{\wt k_n}{k_n}\bigg\|_{L^\infty}\int_D k_n\,d\mu \leq c\,n.
\ee
Therefore, the sampling condition \iref{coptsamplewithc} implies 
$m\geq \gamma \,\|\wt w \,k_n\|_{L^\infty}\ln(2n/\eps)$,
and the results follow by direct application of
Lemma \ref{tropp} and Theorem \ref{cohenmigliorati}. 
\end{proof}

We now concentrate our attention on the offline procedure which should be tuned
in order to ensure that \iref{compk} holds with high probability. For this purpose, we introduce the Gramian matrix 
\be
G_M:=(\<L_j,L_k\>_M)_{j,k=1,\dots n},
\ee
not to be confused with $G$.

\begin{lemma}
\label{firstdraw} For any pair of constants $0<c_1\leq c_2<\infty$, 
the matrix framing property 
\be
c_2^{-1} \,I \leq G_M\leq c_1^{-1} \,I,
\label{matframe}
\ee
implies the uniform framing \iref{compk}.
\end{lemma}

\begin{proof}  We use the fact that, similar to $k_n$, the function $\wt k_n$ is characterized by 
the extremality property
\be
\wt k_n(x)=\max_{v\in V_n}\frac{|v(x)|^2}{\|v\|_M^2}.
\ee
For any $x\in D$ and $v\in V_n$, one has on the one hand
\be
|v(x)|^2 \leq \wt k_n(x)\, \|v\|_M^2 \leq c_1^{-1}\, \wt k_n(x) \,\|v\|^2,
\ee
where the last inequality results from the upper one in \iref{matframe}.
This shows that $c_1\, k_n(x)\leq \wt k_n(x)$. On the other hand, 
using the lower inequality in \iref{matframe}, we find that
\be
|v(x)|^2 \leq k_n(x) \,\|v\|^2 \leq c_2 \,k_n(x) \,\|v\|_M^2,
\ee
which shows that $\wt k_n(x)\leq c_2 \,k_n(x)$. 
\end{proof}

\begin{remark}
The matrix framing \iref{matframe} implies the uniform framing \iref{compk}
but the converse does not seem to hold. Finding an algebraic condition
equivalent to \iref{compk} is an open question.
\end{remark}

Lemma \ref{tropp} indicates that if the amount of offline samples
satisfies the condition
\be
M\geq \gamma \,K_n\ln(2n/\eps), \quad K_n:=\|k_n\|_{L^\infty(D)},
\label{condM1}
\ee
then, we are ensured that
\be
{\Pr}_z\(\|G_M-I\|_2 \geq 1/2\)\leq \eps,
\ee 
and therefore the framing \iref{matframe} holds with 
probability greater than $1-\eps$, for the particular values
$c_1=\frac 2 3$ and $c_2=2$.  Bearing in mind that
$k_n$ is unknown to us, we assume at least that we know an upper estimate
for its $L^\infty$ norm
\be
K_n\leq B(n).
\ee
Explicit values for $B(n)$ for general domains $D$ are established in \S 5
in the case where the $V_n$ are spaces of algebraic polynomials. Therefore, given
such a bound, taking $M$ such that
\be
M\geq \gamma \,B(n)\ln(2n/\eps), 
\label{condM}
\ee
guarantees a similar framing with 
probability greater than $1-\eps$. We obtain the following 
result as a direct consequence of Theorem \ref{theoframe}.

\begin{cor}
\label{corframe}
Assume that the amount of sample $M$ used in the offline stage
described by Algorithm 1 satisfies \iref{condM} for some given
$\eps>0$. Then, under the sampling budget condition 
\be
m\geq 3\,\gamma \,n\,\ln(2n/\eps),
\label{coptsamplewith3}
\ee
for the online stage, the event $E:=\{\|G-I\|_2\leq \frac 1 2 \;{\rm and} \; \|G_M-I\|_2\leq \frac 1 2\}$
satisfies ${\Pr} (E^c) \leq 2\eps$,
In addition, one has the convergence bounds
\be
\E(\|u-\wt u_n\|^2 \Chi_E)\leq (1+\eta(m)) \,e_n(u)^2,
\label{bound332}
\ee
and
\be
\E(\|u-\wt u_n^T\|^2)\leq (1+\eta(m))\, e_n(u)^2+ 8\,\eps \,\tau^2,
\label{bound32}
\ee
where $\eta(m)=12\,\frac{n}{m}\leq\frac {4}{\gamma\ln(2n/\eps)}$.
\end{cor}

\noindent
\begin{proof} The estimate on $\Pr(E^c)$ follows by a union bound. Since
$\|G_M-I\|_2\geq \frac 1 2$ ensures the framing \iref{compk} with $c_1=\frac 2 3$ and $c_2=2$,
the bound \iref{bound332} follows from \iref{bound33} in Theorem \ref{theoframe}. Finally, 
the bound \iref{bound32} follows from \iref{bound332} and the probability estimate on $E^c$ by the same
argument as in the proof of Theorem \ref{cohenmigliorati}.
\end{proof}

\subsection{An empirical determination of the value of $M$}

In many situations, the best available bound $B(n)$ on $K_n$ could be overestimated by a large amount. Moreover, the theoretical requirement $M\geq \gamma \,K_n\ln(2n/\eps)$ is only a sufficient condition
that guarantees that $\|G_M-I\|_2\leq \frac 1 2$ with probability larger than $1-\eps$. It could happen that
for smaller values of $M$, the matrix $G_M$ satisfies the framing \iref{matframe} with constants 
$c_1$ and $c_2$ that have moderate ratio $c=\frac {c_2}{c_1}$.

Since the computational cost of the offline stage is proportional to $M$, it would desirable to use such
a smaller value of $M$. If we could compute the matrix $G_M$
it would suffice to raise $M$ until the condition number
\be
\kappa(G_M)=\frac{\lambda_{\max}(G_M)}{\lambda_{\min}(G_M)},
\ee 
has value smaller than a prescribed threshold $c^*>1$, so that \iref{matframe} holds with $c=\kappa(G_M)\leq c^*$.

However, since the exact orthonormal basis elements $L_j$ are generally unknown to us,
we cannot compute the matrix $G_M$. As an alternate strategy, we propose the following 
method that provides with an empirical determination of $M$: 
start from the minimal value $M=n$, and draw points 
$y^1,\dots, y^M$ and $z^1,\dots z^M$ independently according to $\mu$. Then, defining
\be
\<u,v\>_y=\frac{1}{M}\sum_{i=1}u(y^i)\,v(y^i)\quad\text{and}\quad \<u,v\>_z=\frac{1}{M}\sum_{i=1}u(z^i)\,v(z^i),
\ee
compute an orthonormal basis $(L_j^y)$ with respect to $\<\cdot,\cdot\>_y$, and define the test matrix
\be
T:=(\<L_j^y,L_k^y\>_z)_{j,k=1,\dots,n}.
\ee
If $\kappa(T)\geq c^*$, then raise the value of $M$ by some fixed amount, 
and repeat this step until $\kappa(T)\leq c^*$. For this empirically found value $M=M_{\rm emp}(n)$, use the points $\{z^1,\dots,z^M\}$
in the offline stage described by Algorithm 1, and the constant $c=c^*$ in the sampling budget condition 
\iref{coptsamplewithc} used in the online stage.

The rationale for this approach is that if $G_M$ well conditioned 
with high probability, then $T$ should also be, as shown for example by
the following result.
\begin{prop}
If $M$ is chosen in such a way that $\Pr(\kappa(G_M)\geq c)\leq \eps$ for some $c>1$, then
 \be
 \Pr(\kappa(T)\geq c^2)\leq 2\,\eps.
 \ee
\end{prop}
\begin{proof}
Since both matrices $G_y=(\<L_j,L_k\>_y)_{j,k=1,\dots,m}$ and $G_z=(\<L_j,L_k\>_z)_{j,k=1,\dots,m}$ are
realizations of $G_M$, we obtain by a union bound that,
with probability at least $1-2\eps$, both $G_y$ and $G_z$ have 
condition numbers less than $c$. Under this event,
\begin{align*}
\lambda_{\max}(T)
&=\sup_{\alpha\in \R^n}\frac{\|\sum_{j=1}^n \alpha_j L_j^y\|_z^2}{|\alpha|^2}\leq \sup_{\alpha\in \R^n}\frac{\|\sum_{j=1}^n \alpha_j L_j^y\|^2}{|\alpha|^2}\,\sup_{v\in V_n}\frac{\|v\|_z^2}{\|v\| ^2}\\
&=\(\inf_{v\in V_n}\frac{\|v\|^2}{\|v\|_y^2}\)^{-1}\,\sup_{v\in V_n}\frac{\|v\|_z^2}{\|v\| ^2}=\frac{\lambda_{\max}(G_z)}{\lambda_{\min}(G_y)},
\end{align*}
and
\begin{align*}
\lambda_{\min}(T)
&=\inf_{\alpha\in \R^n}\frac{\|\sum_{j=1}^n \alpha_j L_j^y\|_z^2}{|\alpha|^2}\geq \inf_{\alpha\in \R^n}\frac{\|\sum_{j=1}^n \alpha_j L_j^y\|^2}{|\alpha|^2}\,\inf_{v\in V_n}\frac{\|v\|_z^2}{\|v\| ^2}\\
&=\(\sup_{v\in V_n}\frac{\|v\|^2}{\|v\|_y^2}\)^{-1}\,\inf_{v\in V_n}\frac{\|v\|_z^2}{\|v\| ^2}=\frac{\lambda_{\min}(G_z)}{\lambda_{\max}(G_y)},
\end{align*}
which implies that $\kappa(T)\leq \kappa(G_y)\,\kappa(G_z) \leq c^2$.
\end{proof}

The above proposition shows that
a good conditioning of $G_M$ with high probability 
implies the same property for $T$. There is of course no theoretical guarantee that the value of $M$ provided by the above empirical approach is sufficient to achieve good conditioning of $G_M$, unless the resulting $M$ satisfies \iref{condM}. However, in the numerical experiments of \S 6, we will check that the values of $M$ for which  $\kappa(T)\leq c$ do also ensure that a similar bound holds for $\kappa(G_M)$.

\section{Multilevel strategies}

The sampling strategy that we have outlined in \S 3
provide instance optimal reconstructions of $u$ with
an optimal sampling budget up to a multiplicative factor $\ln(2n/\eps)$. 
Thus, the execution time of the online stage,
dominated by the $m$ evaluations of $u$ at points $x^i$, cannot be significantly improved. 
On the other hand, the complexity of the offline stage is dominated by the computation of the Gramian matrix 
for deriving the basis $\{\wt L_1,\dots,\wt L_n\}$, and is therefore
of order $\cO(Mn^2)$. In particular, it depends 
linearly on the number of points $M$, 
which could be very large if $K_n$ grows fast, or if 
its available bound $B(n)$ is over-estimated. 

In this section we discuss a multilevel approach aiming at improving this
offline computational cost: we produce an approximation 
to $k_n$ in several iterations, by successive refinements 
of this function as the dimension of $V_n$ increases.
We consider a family of nested spaces $(V_{n_p})_{p\geq 1}$ of increasing dimension $n_p$
and take an orthonormal basis $(L_j)_{j\geq 1}$ adapted to this hierarchy, in the sense that 
\be
V_{n_p}=\Span\{L_1,\dots,L_{n_p}\}, \quad n_p\geq 1.
\ee
As previously, the exact functions $k_{n_p}$ are out of reach, since we do not have access 
to the continuous inner product by which we would compute the basis $(L_j)_{1\leq j \leq n_p}$. 
The offline stage described in \S 3 computes approximations $\wt L_j$ by 
orthogonalizing with respect to a discrete inner product with points $z^i$ drawn according to 
$d\mu$. We know that a more efficient sample for performing this 
orthogonalization should be drawn according
to $d\sigma=\frac {k_{n_p}}{n_p} \,d\mu$ which is however unknown to us. The idea for breaking this dependency loop is to
replace $k_{n_p}$ with $\wt k_{n_{p-1}}$,
which was computed at the previous step. Our analysis of this strategy
is based on the following assumption of proximity between $k_{n_{p-1}}$ and $k_{n_p}$: 

There exists a known constant $\kappa>1$ such that
\begin{equation}
\label{hproperty}
 k_{n_1}(x)\leq 3\,\kappa \,n_1 \text{ and } k_{n_p}(x)\leq k_{n_{p+1}}(x) \leq \kappa\, k_{n_p}(x), \quad p\geq 1,\quad x\in D.
\end{equation}
The validity of this assumption can be studied through lower and upper estimates
for $k_n$, such as those discussed in the next section. For example, Theorem \ref{theopu} allows one to establish \iref{hproperty} for bivariate polynomial spaces of total degree $p$, therefore with $n_p=(p+1)(p+2)/2$, on domains with piecewise smooth boundary. Note that \iref{hproperty} allows up to exponential growth of $K_n$, if we simply take $n_p=p$.  

Assuming that the targeted space $V_n$ is a member of this hierarchy, that is, 
\be
n=n_q, \quad  \mbox{for some}\quad q>1,
\ee 
we modify the offline stage as follows.
\nl
\nl
{\bf Algorithm 2.} Start with $\wt w_{0}=1$ and $\wt\sigma_{0}=\mu$. For $p=1,\dots,q$, iterate the following: draw a certain  number $M_{p}$ of points  $z^1_p, \dots, z_p^{M_p}$ independently according to $\wt \sigma_{p-1}$, and construct an
orthonormal basis $(L^p_j)_{1\leq j \leq n_p}$ of $V_{n_p}$
with respect to the inner product
\be
\<u,v\>_p:=\frac{1}{M_p}\,\sum_{i=1}^{M_p}\, \wt w_{p-1}(z^i_p)\,u(z^i_p)\,v(z^i_p).
\label{Mpinner}
\ee
Then define
 \be
 \wt k_{n_p}=\sum_{j=1}^{n_p} |L^p_j|^2,\quad \wt w_{p}=\frac{n_p}{\wt k_{n_p}},\quad\text{and}\quad d\wt \sigma_{p}:=\alpha_{p}\,\frac{\wt k_{n_p}}{n_p}\,d\mu,
 \ee
where $\alpha_{p}$ is the normalization constant, and proceed to the next iteration. At the end of iteration $q$, define the perturbated Christoffel function for $V_n$ as $\wt k_n=\wt k_{n_q}$,
weight function $\wt w=\wt w_q$ and sampling measure 
\be
\wt \sigma=\wt \sigma_q=\alpha\,\frac{\wt k_{n}}{n}\,d\mu,
\ee 
where $\alpha$ is a normalization factor.
\nl

The online stage remains unchanged: the samples $x^1,\dots,x^m$ for evaluation of $u$
are drawn i.i.d. according to $\wt \sigma$, and
we solve the weighed least-squares problem \iref{wls2}.
The sample size $M$ of the offline stage
is now replaced by $\o M=M_1+\dots+M_q$. 
We denote by $G_p:=(\<L_j,L_k\>_p)_{j,k=1,\dots,n_p}$
the Gramian matrices for the inner products \iref{Mpinner},
The following result shows
that the conditions imposed on the $M_p$ are less stringent than 
those that were imposed on $M$.

\begin{theorem}
\label{thmalgo3}
Let $\eps_p >0$ such that $\eps:=\sum_{p=1}^q\eps_p<1$, and assume that the amounts
offline samples in Algorithm 2 satisfy
\be
M_p\geq 3\,  \kappa\,\gamma\, n_p\,\ln \frac{2n_p}{\eps_p},\quad p=1,\dots,q,
\ee
 with $\kappa$ the constant in the assumption \iref{hproperty}.  Then if $m\geq 3\,\gamma\,n\,\ln\frac{2n}{\eps}$, 
the same convergence bounds \iref{bound332} and \iref{bound32} as in Corollary \ref{corframe} hold, with 
$E:=\{\|G-I\|_2\leq \frac 1 2 \;{\rm and} \; \|G_q-I\|_2\leq \frac 1 2\}$ that satisfies $\Pr(E^c)\leq 2\eps$.
\end{theorem}

\begin{proof} We show by induction on $p$ that the event
\be
B_p:=\left\{\|G_1-I\|_2\leq\frac{1}{2},\dots,\|G_p-I\|_2\leq\frac{1}{2}\right\}
\ee
occurs with probability at least $1-\eps_1-\dots-\eps_p$. As 
\be
M_1\geq 3\, \kappa\,\gamma\,n_1\,\ln \frac{2\,n_1}{\eps_1}\geq \gamma \,\|\wt w_0\,k_{n_1}\|_{L^\infty}\,\ln \frac{2\,n_1}{\eps_1},
\ee
 by Lemma~\ref{tropp},
\be
\Pr(B_1)\geq 1-\eps_1.
\ee
For $1\leq p< q$, under the event $B_{p}$, Lemma~\ref{firstdraw} gives
\be
\frac{2}{3}\,k_{n_p}(x)\leq \wt k_{n_p}(x)\leq 2\,k_{n_p}(x),\quad x\in D.
\ee
Therefore, using assumption \iref{hproperty}, we find that
\be
\left\|\frac{\wt w_{p}\, k_{n_{p+1}}}{\alpha_p}\right\|_{L^\infty}=\left\|\frac {n_p} {\alpha_p}\,\frac{k_{n_{p+1}}}{\wt k_{n_p}}\right\|_{L^\infty}\leq n_p\,\bigg\|\frac{\wt k_{n_p}}{k_{n_p}}\bigg\|_{L^\infty}\,\left\|\frac{k_{n_p}}{\wt k_{n_p}}\right\|_{L^\infty}\,\left\|\frac{k_{n_{p+1}}}{k_{n_p}}\right\|_{L^\infty}\leq 3\,\kappa\,n_p.
\ee
As $M_{p+1} \geq 3\,\kappa\,\gamma \,n_p \,\ln \frac{2\,n_p}{\eps_{p+1}}$, Lemma \ref{tropp} applies, and combining this with the induction hypothesis: 
\begin{align*}
\Pr(B_{p+1})&=\Pr(B_p)\, \Pr\(\|G_{p+1}-I\|_2\leq\frac{1}{2}\,\Big|\,B_p\)\\
&\geq (1-\eps_1-\dots -\eps_p)(1-\eps_{p+1})\geq1-\eps_1-\dots -\eps_p-\eps_{p+1}.
\end{align*}
Use Lemma \ref{firstdraw} one last time to write, in the event $B_{q}$,
\be
\frac{2}{3}\,k_{n_q}(x)\leq \wt k_{n_q}(x)\leq 2\,k_{n_q}(x),\quad x\in D,
\ee
which is the framing \iref{matframe} for the particular values
$c_1=\frac 2 3$ and $c_2=2$. Since $B_q$ has probability larger than $1-\eps$,
we conclude by the exact same arguments used in the proof Corollary \ref{corframe}.
\end{proof}

We now comment on the gain of complexity by using Algorithm 2:
\begin{enumerate}
\item
Exponential growth of $K_n$: the property \iref{hproperty} might be satisfied even when $K_n$ grows exponentially
with $n$, by taking the choice $n_p=p$. Then, the complexity of Algorithm 1 
is of order $\cO(M\,n^2)\gsim \cO(K_n\,n^2\ln(2n/\eps))$, 
which grows exponentially in $n$. In contrast, 
the total amount of sampling in Algorithm 2 is $\o M=M_1+\dots+M_n\leq n\,M_n=\cO(n^2\ln(2n/\eps))$,
 so the first stage remains of polynomial complexity $\cO(n^4\ln(2n/\eps))$.
\item
Algebraic growth of $K_n$:  if $K_n\sim n^r$ only grows algebraically in $n$, one may choose $n_p=2^p$, in which case the total number of sample points $\o M$ rewrites as $M_{n_0}+\dots+M_{n_q}\sim M_{n_q}$, giving an optimal complexity $\cO(n^3\ln(2n/\eps))$ for the first stage. This is smaller than the complexity $\cO(K_n \,n^2\ln(2n/\eps))=\cO(n^{2+r}\ln(2n/\eps))$ encountered in Algorithm 1.
\end{enumerate}

While Algorithm 2 produces a computational gain in computing a near-optimal measure $\wt \sigma$, 
the resulting sample $x^1,\dots,x^m$ is specifically targeted at approximating $u$
in the space $V_n=V_{n_q}$. As explained in \S \ref{optbench}, it is sometimes desirable to obtain
optimal weighted least-squares approximations $\wt u_{n_p}$ for each space $V_{n_p}$
while maintaining the cumulated number of evaluations of $u$ until step $p$ of the optimal order $n_p$
up to logarithmic factors. Therefore, 
we would like to recycle the evaluation points $\{x^1,\dots,x^{m_{p-1}}\}$ used until step $p-1$
in order to create the new evaluation sample $\{x^1,\dots,x^{m_{p}}\}$, for some well chosen sequence $(m_p)_{p\geq 1}$
that grows similar to $(n_p)_{p\geq 1}$ up to logarithmic factors.

 Intuitively, since the sample should have a density proportional to $k_{n_p}$,
most of the new points we draw at step $p$ should be distributed according a density 
proportional to $k_{n_p}-k_{n_{p-1}}=\sum_{j=n_{p-1}+1}^{n_p}|L_j|^2$. This leads us
to the following algorithm.
\nl
\nl
{\bf Algorithm 3.} Start with $\wt w_{0}=1$ and $\wt\sigma_{0}=\mu$ and $m_0=0$. For $p=1,2,\dots$, generate
$z^i_{n_p}$ and compute $\wt w_{n_p}$, $\wt \sigma_{n_p}$ and $\wt k_{n_p}$ as in Algorithm 2. 
When creating the orthonormal basis $(L_j^{n_p})$, ensure compatibility with the inclusion 
$V_{n_{p-1}}\subset V_{n_p}$, in the sense that 
\be
\Span(L_1^{n_p},\dots,L_{n_{p-1}}^{n_p})=V_{n_{p-1}}.
\ee
Having defined the evaluation points $\{x^1,\dots,x^{m_{p-1}}\}$, 
draw the new evaluation points $x^i$ for $i=m_{p-1}+1,\dots,m_p$ according to
\be
\label{rhop}
d\rho_p:=\frac{\alpha_p}{m_p-m_{p-1}}\(\frac{m_p}{n_p}\,\sum_{j=1}^{n_p}|L_j^{n_p}|^2-\frac{m_{p-1}}{n_{p-1}}\,\sum_{j=1}^{n_{p-1}}|L_j^{n_p}|^2\)\,d\mu,
\ee
with $\alpha_p$ a normalization factor.
\nl

\begin{remark}
 Note that the non-negativity of $\rho_p$ is only guaranteed when $(m_p/n_p)_{p\geq1}$ is non-decreasing, a condition which is easily met since $m_p$ has to grow as $n_p\,\ln n_p$. If we had taken $m_p$ exactly linear with respect to the dimension $n_p$, the terms with $j\leq n_{p-1}$ in the expression \iref{rhop} would cancel, hence $d\rho_p$ would only be an approximation of $\frac{k_{n_p}-k_{n_{p-1}}}{n_p-n_{p-1}}\,d\mu$.
\end{remark}

\begin{remark}
In the above algorithm, the various sections $\{x^{m_{k-1}+1},\dots,x^{m_k}\}$
of $\{x^1,\dots,x^{m_p}\}$ for $k=1,\dots,p$ are drawned
according to different probability measures. The sample $\{x^1,\dots,x^{m_p}\}$ 
is thus not i.i.d. anymore, which affects the proof of the convergence theorem
given below. Instead it may be thought as a deterministic mixture
of a collection of i.i.d. samples, as introduced and studied in \cite{Mig2}. 
\end{remark}

At any iteration $q$, we use the evaluations of $u$ at all points $x^1,\dots,x^{m}$
as follows to compute a least-squares approximation $\wt u_{n}\in V_{n}$, where $n:=n_q$ and $m:=m_q$. We denote
by $w$ the weight function defined by
\be
w(x)\,\sum_{p=1}^q(m_p-m_{p-1})\,d\rho_p=m\,d\mu,
\label{defw}
\ee
and solve the weighted least square problem \iref{wls}.
The following result shows that
instance optimality is maintained at every step $q$, with a cumulated sampling budget $m_q$ that is
near-optimal.

\begin{theorem}
\label{thmalgo4}
Take numbers $\delta_p,\eps_p\in ]0,1[$ such that $\eps:=\sum_{p=1}^q\eps_p<1$ and $\delta:=\sum_{p=1}^q\delta_p< 1/2$, and define $c_{\delta}=((1+\delta)\ln\, (1+\delta)-\delta)^{-1}$. Assume that, for all $p\geq 1$,
\be
M_{n_p}\geq 2\, \kappa\,c_{\delta_p}\, n_p \ln\frac{2 n_p}{\eps_p}\quad \text{and} \quad m_p\geq \frac{\gamma}{1-2\delta}\,n_p\ln \frac{2n_p}{\eps},
\ee
with $\kappa$ the constant in the assumption \iref{hproperty}, and that $m_p/n_p$ is an non-decreasing function of $p$. Then, with $n:=n_q$ and $m:=m_q$, the convergence bounds \iref{bound332} and \iref{bound32} 
simultaneously hold for all $q\geq 1$, with
\be
\eta(m)=\frac{4}{(1-2\delta)}\frac{n}{m}\leq\frac {4}{\gamma\ln(2n/\eps)}
\ee
and $E:=\{\|G-I\|_2\leq \frac 1 2\; {\rm and} \;\|G_p-I\|_2\leq \delta_p, \, p\geq 1\}$, which satisfies $\Pr(E^c)\leq 2\eps$.
\end{theorem}

The proof of this theorem requires a refinement of Lemma \ref{tropp}, due to the fact that 
the $x^i$ are not anymore identically distributed. This uses the following tail bound, directly obtained from the matrix Chernoff bound in \cite{Tr}.

\begin{prop}
\label{proptr}
 Consider a finite sequence $\{X^i\}_{i=1,\dots,m}$
 of independent, random, self-adjoint matrices with dimension $n$. Assume that each matrix satisfies $0\leq X^i\leq R\,I$ almost surely, and that $\sum_{i=1}^m \E (X^i)=I$. Then for all $\delta\in ]0,1[$,
 \be
 \Pr\(\,\left\|\,\sum_{i=1}^m X^i-I\,\right\|_2> \delta\)\leq 2\,n\exp\(-\frac1{c_\delta\,R}\),
 \ee
where $c_{\delta}=((1+\delta)\ln\, (1+\delta)-\delta)^{-1}$ as in Theorem \ref{thmalgo4}.
\end{prop}

\noindent
{\bf Proof of Theorem \ref{thmalgo4}:} By the same argument as in Theorem \ref{thmalgo3},
we find that the event 
\be
B=\{\|G_p-I\|_2\leq \delta_p, \quad p\geq 1\}
\ee
has probability larger than $1-\eps$, where the $G_p$ are as in the proof of Theorem \ref{thmalgo3}.

 We then fix a value of $q$ and for $n=n_q$ and $m=m_q$, we study the Gramian matrix
$G$ which is the sum of the independent, but not identically distributed, matrices
\be
X^i:=\frac{1}{m}\,w(x^i) \,(L_j(x^i)L_k(x^i))_{j,k=1,\dots,n}, \quad i=1,\dots,m.
\ee
Then, with the notation $H(x)=(L_j(x)L_k(x))_{j,k=1,\dots,n}$,
\be
\sum_{i=1}^m\, \E (X^i)= \sum_{p=1}^q\, (m_p-m_{p-1})\,\int_D\frac{1}{m}\,w(x)\,H(x)\,d\rho_p(x)=\int_D H(x)\,d\mu(x)=I, 
\ee
and 
\be
\|X^i\|_2=\frac{1}{m}\,w(x^i)\,\sum_{j=1}^n |L_j(x^i)|^2 \leq \frac{1}{m}\,\|w\,k_n\|_{L^{\i}}=:R.
\ee
One also has, under the event $B$, $\int_D |L_j^{n_p}|^2\,d\mu\leq\frac{1}{1-\delta_p}$ for $j=1,\dots,n_p$ so $\alpha_p \geq 1-\delta_p$, and consequently
\begin{align*}
 \frac{m}{w}&=\sum_{p=1}^q\, (m_p-m_{p-1})\,\frac{d\rho_p}{d\mu}\\
 &=\sum_{p=1}^q\alpha_p\(\frac{m_p}{n_p}\,\sum_{j=1}^{n_p}\,|L_j^{n_p}|^2-\frac{m_{p-1}}{n_{p-1}}\,\sum_{j=1}^{n_{p-1}}\,|L_j^{n_p}|^2\)\\
 &\geq \sum_{p=1}^q \(\frac{m_p}{n_p}\,\frac{1-\delta_p}{1+\delta_p}\,k_{n_p}-\frac{m_{p-1}}{n_{p-1}}\,k_{n_{p-1}}\)\\
 &\geq \frac{m}{n}\,k_{n} -\sum_{p=1}^{q} \frac{m_p}{n_p}\,\frac{2\delta_p}{1+\delta_p}\,k_{n_p}\\
 & \geq (1-2\delta)\,\frac{m}{n}\,k_{n},
\end{align*}
so $R=\frac{1}{m}\,\|w\,k_{n}\|_{L^\infty}\leq \frac{1}{(1-2\delta)}\,\frac{n}{m}$. Applying Proposition \ref{proptr}, we find that
\be
{\Pr}_x\(\|G-I\|_2>\frac{1}{2} \; \Big | \; B \)\leq 2\,n\,\exp\(-\frac1{\gamma\,R}\)\leq 2\,n\,\exp\(-\frac{1-2\delta}{\gamma}\,\frac{m}{n}\)\leq\eps.
\ee
Therefore, since $E:=B\cap \left\{ \|G-I\|_2\leq\frac{1}{2}\right\}$, we find that $\Pr(E)\geq 1-2\eps$. 
 
 In order to prove the convergence bounds \iref{bound332} and \iref{bound32},
 we cannot proceed as in Corollary \ref{corframe} by simply invoking Theorem \ref{theoframe}, because
 the $x^i$ are not identically distributed. This leads us to modify the statement of 
 Lemma \ref{lemcomputation} and its proof given in \cite{CM1}. 
 First, using similar arguments as in  \cite{CM1}, we find that
 \be
 \E(\|u-\wt u_n\|^2 \,\Chi_E) \leq e_n(u)^2+4\,\E \(\sum_{k=1}^n|\<L_k,g\>_m|^2\,\Chi_{E}\),
 \label{beginningoftheproofofCM}
 \ee
 where $g=u-P_nu$ is the projection error.  For each $k=1,\dots,n$,
we define $g_k:=w\,L_k\,g$ and write
  \begin{align*}
\E\,(|\<L_k,g\>_m|^2\,\Chi_{E}) &\leq \E\,(|\<L_k,g\>_m|^2\,\Chi_{B})\\
&=\frac{1}{m^2}\sum_{1 \leq i,j \leq m} \E\,\(g_k(x^i)\,g_k(x^j)\,\Chi_{B}\)\\
&=\frac{1}{m^2}\,\E_z\(\Chi_{B}\(\sum_{1\leq i \leq m}\E_x \(|g_k(x^i)|^2\)+\sum_{i\neq j}\,\E_x\(g_k(x^i)\,g_k(x^j)\)\)\)\\
&\leq\frac{1}{m^2}\,\E_z\( \Chi_{B}\(\sum_{1\leq i \leq m}\E_x \(|g_k(x^i)|^2\)+\(\sum_{1\leq i \leq m}\,\E_x\(g_k(x^i)\)\)^2\,\)\)\\
&=\E_z\(\Chi_{B}\(\frac{1}{m}\,\E_t\(|g_k(t)|^2\)+\big(\,\E_t\,(g_k(t))\big)^2\)\),
\end{align*}
where $t$ is a random variable distributed according to $\sum_{p=1}^q \frac{m_p-m_{p-1}} m\,d\rho_p=\frac{1}{w}\,d\mu$. 
We then note that
\be
\E_t(g_k(t))=\int_D g\,L_k\, d\mu=0
\ee
since $g\in V_n^\perp$, and that $\sum_{k=1}^n|g_k(t)|^2=w(t)^2 \,g(t)^2\,k_n(t)$.
Therefore 
\begin{align*}
\E \(\sum_{k=1}^n|\<L_k,g\>_m|^2\,\Chi_{E}\) & \leq \E_z\(\Chi_{B} \, \frac 1 m \int_D w\,k_n \,g^2 d\mu\)\\
& \leq  \E_z\(\Chi_{B} \,R \,\|g\|^2\) \\
& \leq \frac{1}{(1-2\delta)}\,\frac{n}{m}\,e_n(u)^2
\end{align*}
Combining this with \iref{beginningoftheproofofCM}, we finally obtain
\be
\E(\|u-\wt u_n\|^2 \,\Chi_E) \leq \(1+\frac{4}{(1-2\delta)}\,\frac{n}{m}\)e_n(u)^2
\ee
\hfill $\Box$ 

\begin{remark}
If a stopping time $q$ is known in advance, the simplest choice is to take  $\eps_p=\eps/q$ and $\delta_p=\delta/q$.
If the stopping time $q$ is not known in advance, we can take for instance $\eps_p=\frac{6}{\pi^2}\,\frac{\eps}{p^2}$ and $\delta_p=\frac{6}{\pi^2}\,\frac{\delta}{p^2}$. As $c_\delta\sim \frac{2}{\delta^2}$ when $\delta\rightarrow0$, this choice only increases the number $M_p$ of sample points $z^i$ by a factor $p^4$, which is satisfying in view the previous remarks.
\end{remark}

\section{Estimates on the inverse Christoffel function}

We have seen that the success of Algorithm 1 is based
on the offline sampling condition \iref{condM}, which
means that a uniform upper bound  $B(n)$ on 
the inverse Christoffel function $k_n$ 
is needed in the first place. Likewise, the multilevel 
Algorithms 2 and 3 from \S 4 are based on the assumption \iref{hproperty},
which verification requires pointwise upper and lower estimates on $k_n(x)$.
In this section we establish such bounds and pointwise
estimates on general domains when the $V_n$ are spaces of algebraic multivariate
polynomials of varying total degree. Throughout this section, we assume that 
\be
\mu=\mu_D=|D|^{-1}\,\Chi_D\,dx
\ee 
is the uniform measure over $D$, which is
thus assumed to have finite Lebesgue measure $|D|$. 

\subsection{Comparison strategies}

Our vehicle for estimating the Christoffel function is a general strategy, 
first introduced in \cite{Kro}: compare $D$ with reference domains 
$R$ for which the Christoffel function can be estimated. For simplicity,
we use the notation 
\be
L^2(R)=L^2(R,\mu_R),
\ee 
for any domain $R$ where $\mu_R=|R|^{-1}\,\Chi_R\,dx$ is the uniform measure over $R$. In order to make clear the dependence on the domain, we define
\be
k_{n,R}(x)=\max_{v\in V_n}\frac{|v(x)|^2}{\;\|v\|^2_{L^2(R)}\!},
\ee
and
\be
K_{n,R}=\|k_n\|_{L^\infty(R)}=\max_{v\in V_n}\frac{\,\|v\|_{L^\infty(R)}^2}{\|v\|^2_{L^2(R)}},
\ee
We first state a pointwise comparison result.

\begin{lemma}
\label{lemmaref2}
For $x\in D$, let $R$ be such that $x\in R\subset D$ and $\beta\, |D|\leq |R|$ for some $\beta \in ]0,1]$. Then
\be
k_{n,D}(x)\leq  \beta^{-1}\,k_{n,R}(x).
\ee
Conversely, let $S$ be such that $D\subset S$ and $\beta \,|S|\leq |D|$ for some $\beta \in ]0,1]$. Then
\be
 k_{n,D}(x) \geq \beta\, k_{n,S}(x).
\ee
\end{lemma}

\begin{proof}
For any $v\in V_n$, we have
\be
|v(x)|^2\leq k_{n,R}(x) \,\|v\|_{L^2(R)}^2\leq k_{n,R}(x)\, \frac{|D|}{|R|}\,\|v\|_{L^2(D)}^2,
\ee
and
\be
|v(x)|^2\leq k_{n,D}(x) \,\|v\|_{L^2(D)}^2\leq k_{n,D}(x)\, \frac{|S|}{|D|}\,\|v\|_{L^2(S)}^2.
\ee
Optimizing over $v$ gives the upper 	and lower estimates of $k_{n,D}(x)$. 
\end{proof}
Obviously, a framing on $K_{n,D}$ can be readily derived as follows, by
application of the above lemma to any point in $D$.
\begin{prop}
\label{lemmaref}
Assume that there exist a family $\cR$ of reference domains with the following 
properties:
\begin{itemize}
\item[(i)] For all $x\in D$ there exist $R_x\in\cR$ such that $x\in R_x\subset D$.
\item[(ii)] There exists a constant $\beta\in ]0,1]$ such that $|R|\geq \beta\,|D|$ for all $R\in\cR$.
\end{itemize}
Then, one has 
\be
K_{n,D}\leq \beta^{-1}\sup_{x\in D} k_{n,R_x}(x) \leq \beta^{-1}\sup_{R\in \cR} K_{n,R}.
\ee
Likewise, for any $S\in \cR$ such that $D\subset S$ and 
$|D|\geq \beta\,|S|$, one has 
\be
K_{n,D} \geq \beta \sup_{x\in D} k_{n,S}(x).
\ee
\end{prop}

In what follows, we apply this strategy to spaces $V_n$ of multivariate algebraic polynomials.
Throughout this section, we consider 
\be
V_n=\P_\ell:={\rm span} \{x \mapsto x^\nu=x_1^{\nu_1}\dots x_d^{\nu_d} \; : \; |\nu|=\nu_1+\dots+\nu_d\leq \ell\},
\label{total}
\ee
the space of polynomials with total  degree less or equal to $\ell$, for which we have 
\be
n={d+\ell \choose \ell}.
\ee
We assume $D$ is a bounded open set of $\mathbb R^d$.

It is important to note that $V_n$ is invariant by affine transformation. As a consequence, if $A$ is
any affine transformation, one has
\be
R'=A(R) \implies k_{n,R'}(A(x))=k_{n,R}(x), \quad x\in R,
\ee
and in particular $K_{n,R'}=K_{n,R}$.

\subsection{Lipschitz domains}

In the case of the cube $Q=[-1,1]^d$, we may express $k_{n,Q}$
by using tensorized Legendre polynomials, that is
\be
k_{n,Q}(x)=\sum_{|\nu|\leq \ell} |L_\nu(x)|^2, \quad L_\nu(x)=\prod_{i=1}^d L_{\nu_i}(x_i),
\ee
where the univariate polynomials $t\mapsto L_j(t)$ are normalized in $L^2([-1,1],\frac {dt} 2)$. 
Using this expression, it can be proved by induction on the dimension $d$ that 
\be
K_{n,Q}\leq n^2, \quad n\geq 1,
\ee 
see Lemma 1 in \cite{CCMNT}. Therefore, by affine invariance,
\be
K_{n,R}\leq n^2, \quad n\geq 1,
\label{para}
\ee
for all 
$d$-dimensional parallelogram $R$.  Using this result, we may
bound the growth of Christoffel functions from above for a general class of domains.

\begin{definition}
An open set $D\subset \R^d$ satisfies the inner cone condition if there exist
$\bar r>0$ and $\theta\in (0,\pi)$, such that for all $x\in \o D$, there exists a unit vector $u$ such that the cone
\be
C_{\bar r,\theta}(x,u):=\{x+r\,v,\, 0\leq r \leq \bar r,\,|v|=1,\, u\cdot v \geq \cos(\theta)\}
\ee
is contained in $\o D$. In particular, any Lipschitz domain $D\subset \R^d$ satisfies the inner cone condition.
\end{definition}

\begin{theorem}
Let $D\subset \R^d$ be a bounded domain that satisfies the inner cone condition.
Then, one has 
\be
K_n\leq C_D\,n^{2}, \quad n\geq 1,
\label{boundKnlip}
\ee 
where $C_D$ depends on $d$, $|D|$, and on $\bar r$ and $\theta$ in the previous definition.
\end{theorem}

\begin{proof}
 The uniform cone condition ensures
that there exists $\kappa=\kappa(\bar r,\theta,d)>0$ such that for any $x\in D$, there exist a parallelogram $R$
such that $x\in R\subset D$ and $|R|=\kappa$. Therefore, applying Proposition \ref{lemmaref} with $\cR$ the family of all parallelograms of area $\kappa$, one
obtains \iref{boundKnlip} with $C_D=\frac {|D|}{\kappa}$.
\end{proof}

\begin{remark}
The bound $K_{n,Q}\leq n^2$ is actually established in \cite{CCMNT} for the more general class
of polynomial spaces of the form
\be
V_n=\P_\Lambda:={\rm span}\{x\mapsto x^\nu\; : \; \nu\in \Lambda\}, \quad \#(\Lambda)=n,
\ee
where $\Lambda\in\N^d$ is downward closed, i.e. such that 
\be
\nu\in \Lambda\quad {\rm and} \quad \wt \nu\leq \nu \implies \wt \nu\in \Lambda.
\ee
These spaces are however not invariant by affine transformation and so 
one cannot apply the above method to treat general domains with inner cone condition. 
On the other hand, these spaces are invariant by affine transformation of the form
$x\mapsto  x_0+ M x$ where $M$ is a diagonal matrix, therefore transforming the cube $Q$
into an arbitrary rectangle $R$ aligned with the coordinate axes. As observed in \cite{AH},
this leads to a bound of the form \iref{boundKnlip} for any domain $D$ that
satisfies the following geometrical property: for all $x\in D$ there exists a rectangle $R$ aligned
with the coordinate axes such that $x\in R\subset D$ and $|R|\geq \beta\, |D|$. Note that  this property does not readily follows from a smoothness property of the boundary, in particular there exists smooth domains for which this property does not hold.
\end{remark}

\subsection{Smooth domains}

We next investigate smooth domains. For this purpose, we
replace parallelograms by ellipsoids as reference domains. 
In the case of the unit ball $B:=\{|x|\leq 1\}$, it is known \cite{Xu} that
the Christoffel function reaches its maximum on the 
unit sphere $S:=\{|x|=1\}$, where we have
\be
k_{n,B}(x)={\ell+d+1 \choose \ell}+{\ell+d-2 \choose \ell-1}.
\ee
In order to estimate how this quantity scales with $n={\ell+d\choose \ell}$
we use the fact that for any integer $m$, one has  
\be
e\left(\frac m e\right)^m\leq m! \leq m^m.
\ee
For the lower bound, we bound from below the first term
\begin{align*}
{\ell+d+1 \choose \ell}&={\ell+d \choose \ell}\,\frac{\ell+d+1}{d+1}=n\,\frac{\ell+d+1}{d+1}\\
&\geq \frac{n}{d\,e^{1/d}}(\ell+d+1) \geq  \frac{n}{e\,(d!)^{1/d}}\,\left(\frac{(\ell+d)!}{\ell!}\right)^{1/d}=e^{-1}\,n^{\frac{d+1}d},
\end{align*}
which leads to
\be
K_{n,B}\geq k_{n,B}(x) \geq e^{-1}\,n^{\frac {d+1}{d}}, \quad x\in S.
\label{knsphere}
\ee
For the upper bound, we write 
\begin{align*}
{\ell+d+1 \choose \ell}+{\ell+d-2 \choose \ell-1}&={\ell+d \choose \ell} \(\frac{\ell+d+1}{d+1}+\frac{\ell d}{(\ell+d)(\ell+d-1)}\)\\
& \leq n \(\frac{\ell+d+1}{d+1}+1\)=n\left(\frac{\ell}{d+1}+2\right),
\end{align*}
Since 
\be
n^{1/d}=(d!)^{-1/d} \left(\frac{(\ell+d)!}{\ell!}\right)^{1/d} \geq \frac {\ell+1}{d}\geq\frac{\ell}{d+1},
\ee
we find that
\be
k_{n,B}(x) \leq  \left(n^{1/d}+2\right)n\leq 3\,n^{\frac {d+1}{d}}.
\ee
By affine invariance, we thus obtain 
\be
e^{-1}\,n^{\frac {d+1}{d}}\leq K_{n,E}\leq 3\,n^{\frac {d+1}{d}},
\label{ellipsoid}
\ee 
for all ellipsoids $E$.  This leads to the following result.

\begin{theorem}
\label{upper bound smooth}
Assume $D\subset \R^d$ is a bounded domain with $\cC^2$ boundary.
Then, one has 
\be
K_n\leq C_D\, n^{\frac {d+1}{d}}, \quad n\geq 1, 
\label{Knupsmooth}
\ee
where $C_D$ depends on $D$.
\end{theorem}

\noindent
\begin{proof} Since the boundary of $D$ has finite curvature, we are ensured 
that there exists a $\beta>0$ such that for any $x\in D$, 
there exist an ellipsoid $E$
such that $x\in E\subset D$ and $|E|\geq \beta\, |D|$. Therefore, applying
Proposition \ref{lemmaref} with $\cR$ the family of ellipsoids 
with area larger than $\beta\, |D|$, we obtain \iref{Knupsmooth} with $C_D=3\,\beta^{-1}$. 
\end{proof}

\begin{remark}
In the above argument, one could simply use balls instead of ellipsoids, however at the price
of diminishing the value of $\beta$ and thus raising the constant $C_D$.
\end{remark}

We next give a general lower bound for $K_n$ showing that the above rate
for smooth domains is sharp.

\begin{theorem}
\label{lower bound smooth}
Let $D\subset \R^d$ be an arbitrary bounded domain, and let $B$ be its Chebychev ball,
that is, the smallest closed ball that contains $D$.
Then, one has 
\be
K_{n,D}\geq e^{-1}\,\frac{|B|}{|D|}\,n^{\frac {d+1}{d}}, \quad n\geq 1.
\ee
\end{theorem}

\noindent
\begin{proof} As $\o D$ is compact and $B$ is the smallest possible ball containing $\o D$, there exists a point $x\in \o D\cap \partial B$, and by Lemma \ref{lemmaref2} one has 
\be
K_{n,D}\geq k_{n,D}(x)\geq \frac{|D|}{|B|}\,k_{n,B}(x) \geq e^{-1}\,\frac{|B|}{|D|}\,n^{\frac {d+1}{d}},
\ee
where the last inequality follows from \iref{knsphere} and affine invariance.
\end{proof}

\subsection{Pointwise bounds for piecewise smooth domains}

As already observed, it may be needed to get sharper bounds on $k_n(x)$ that depend on the 
point $x$, in particular when checking the validity of \iref{hproperty}. In the case of algebraic polynomials in dimension
$d=2$, so $n=\frac{(\ell+1)(\ell+2)}{2}$, such bounds have been obtained 
for a particular class of piecewise smooth domains with exiting corners, 
in the following result from \cite{Pry}.

\begin{theorem}
\label{theopu}
Let $D\subset \R^2$ be a bounded open such that $\partial D=\cup_{i=1}^K \Gamma_i$, where the $\Gamma_i$ are one-to-one $C^2$ curves that intersect only at their extremities, at which points 
the interior angles belong to $(0,\pi)$.  Then, there exists a constant $C_D$ that only depends on $D$
such that, for all $x\in D$, 
\be
C_D^{-1}\,k_{n}(x)\leq n\,\underset{(i,j)\in S}{\min} \,\rho_i(x)\,\rho_j(x) \leq C_D\,k_{n}(x), \quad n\geq 1,
\ee
where $S$ consists of the $(i,j)$ such that $\Gamma_i$ and $\Gamma_j$ intersects, and
$\rho_i(x):=\min\(\ell,d(x,\Gamma_i)^{-1/2}\)$.
\end{theorem}

For the square domain $D=Q=[-1,1]^2$, this implies that $k_{n,Q}(x)\sim n\, \ell^2\sim n^2$ when $x$ is close enough to a corner, and we retrieve the bound $K_{n,Q} \leq C_D \,n^2$ from \iref{boundKnlip}. It is also proved that $k_n(x)\sim\,n\, \min\(\ell,d(x,\partial D)^{-1/2}\)$ for bidimensional domains with $\cC^2$ boundary,
which is consistent with the global bound \iref{Knupsmooth} in the case $d=2$.

\subsection{Rate of growth of $K_{n,D}$ and order of cuspitality}

We end this section by a more technical but systematic approach which allows us
to estimate the rate of growth of the inverse Christoffel function in a sharp way
for domains $D$ that could either be smooth, of $\alpha$-H\"older boundary, or even
with cusps of a given order. It is based on using the following more elaborate reference domain that describes a
certain order of smoothness at the origin.

\begin{definition}
For $\alpha_1,\dots,\alpha_{d-1}\in ]0,2]$, denote $R_{\alpha_1,\dots,\alpha_{d-1}}$ the reference domain
\be
R_{\alpha_1,\dots,\alpha_{d-1}}:=\left\{x\in \mathbb [-1,1]^d, \ \max_{1\leq i \leq d-1} |x_i|^{\alpha_i}\leq x_d\right\}.
\ee
\end{definition}

We shall establish upper and lower bounds for $K_{n,D}$ based
on comparisons between $D$ and affine transformations of this reference domain,
by adapting certain techniques and results from \cite{DP}.
The upper bound is as follows.

\begin{theorem}
\label{upper bound R alpha}
Let $D$ be a bounded domain. Assume there exist $\alpha_1,\dots,\alpha_{d-1}\in ]0,2]$ and $\beta>0$ such that, for all $x\in D$, one can find an affine map $A$ such that $A(0)=x$, $A(R_{\alpha_1,\dots,\alpha_{d-1}})\subset \o D$ and $|A(R_{\alpha_1,\dots,\alpha_{d-1}})|\geq \beta\,|D|$. Then
\be
K_{n,D}\leq C_D\, n^{\frac{1}{d}\left(2+\sum_{i=1}^{d-1}2/ {\alpha_i}\right)},
\label{eqn upper bound R alpha}
\ee
where $C_D$ is a constant depending only on $D$.
\end{theorem}

This result is obtained with the extension strategy proposed in \cite{DP}, which consists in combining Proposition \ref{extension thm} below with a comparison of domains. Such a method was applied in the same paper to the case of smooth domains, polytopes, some 2-dimensional domains, and $l^\alpha$ balls in $\mathbb R^d$, which all correspond to the situation $\alpha_1=\dots=\alpha_{d-1}\in [1,2]$ in our theorem. We give below a series of intermediate
results that lead to the proof of Theorem \ref{upper bound R alpha}.

\begin{lemma}
\label{x2 xalpha}
For $\alpha\in ]0,2]$ and $n\geq 1$, the function $f:x\mapsto \frac{1}{9\ell^2}+\beta x^2-|x|^\alpha$ remains non-negative on $\mathbb R$ as soon as $\beta\geq \frac{\alpha}{2}\left(\frac{9}{2}\,(2-\alpha)\,\ell^2\right)^{\frac{2-\alpha}{\alpha}}$.
\end{lemma}

\begin{proof}
As $f$ is symmetric, one only has to consider this function on $\mathbb R_+$. For $x>0$, $f'(x) =2\beta x-\alpha x^{\alpha-1}$ cancels only at $x_0=\left(\frac{\alpha}{2\beta}\right)^{\frac{1}{2-\alpha}}$, so
\be
\min_{x\in \mathbb R} f(x)=f(x_0)=\frac{1}{9\ell^2}-\frac{2-\alpha}{2}\left(\frac{\alpha}{2\beta}\right)^{\frac{\alpha}{2-\alpha}},
\ee
which is non-negative if and only if $\beta\geq \frac{\alpha}{2}\left(\frac{9}{2}\,(2-\alpha)\,\ell^2\right)^{\frac{2-\alpha}{\alpha}}$.
\end{proof}

The following result is Theorem 5.2 from \cite{DP}.

\begin{prop}
\label{extension thm}
 Suppose $D\subset \mathbb R^d$ is a compact set and $T$ is an affine transformation of $\mathbb R^d$ such that $T(B(0,1))\subset D$. Then
 \be
 k_{n,D}\left(T\left(0,\dots,0,1+\frac{1}{3\ell^2}\right)\right)\leq c\,|\det T|^{-1}\ell^{d+1}.
 \ee
where $c$ depends only on $d$.
\end{prop}

\begin{lemma}
For $\alpha_1,\dots,\alpha_{d-1}\in ]0,2]$, one has 
\be
k_{n,R_{\alpha_1,\dots,\alpha_{d-1}}}(0)\leq C \ell^{2+\sum_{i=1}^{d-1}2/ {\alpha_i}},
\ee
where $C$ depends only on $d$.
\end{lemma}

\begin{proof}
Define $\beta_i=\frac{\alpha_i}{2}\left(\frac{9}{2}\,(2-\alpha_i)\,\ell^2\right)^{\frac{2-\alpha_i}{\alpha_i}}$ for $1\leq i \leq d-1$, and let $T$ be the affine transformation
\be
T:x=(x_1,\dots,x_d)\mapsto \left(\frac{x_1}{\sqrt{3\beta_1}},\dots,\frac{x_{d-1}}{\sqrt{3\beta_{d-1}}},\frac{1}{3}\left(1+\frac{1}{3\ell^2}-x_d\right)\right).
\ee
Then, for all $x\in B(0,1)$, $T(x)_d\in [0,1]$ and $1\leq i \leq d-1$, using Lemma \ref{x2 xalpha},
\be
T(x)_d= \frac{1}{3}\left(1+\frac{1}{3\ell^2}-x_d\right)\geq  \frac{1}{3}\left(\frac{1}{3\ell^2}+x_i^2\right)=\frac{1}{9\ell^2}+\beta_i T(x)_i^2\geq T(x)_i^{\alpha_i},
\ee
so $\max_{1\leq i \leq d-1} |T(x)_i|^{\alpha_i}\leq T(x)^d$, which implies that $T(B(0,1))\subset R_{\alpha_1,\dots,\alpha_{d-1}}$.

As $T\left(0,\dots,0,1+\frac{1}{3\ell^2}\right)=0$, a direct application of Proposition \iref{extension thm} gives
\begin{align*}
k_{n,R_{\alpha_1,\dots,\alpha_{d-1}}}(0)
&\leq c\,|\det T|^{-1}\ell^{d+1}
=3\,c\prod_{i=1}^d\sqrt{3\beta_i}\,\ell^{d+1}\\
&\leq C\ell^{d+1+\sum_{i=1}^{d-1}\frac{2-\alpha_i}{\alpha_i}}
=C\ell^{2+\sum_{i=1}^{d-1}\frac{2}{\alpha_i}}
\end{align*}
\end{proof}

\noindent
{\bf Proof of Theorem \ref{upper bound R alpha}:} One simply applies Proposition \ref{lemmaref} to the family $\cR$ of all domains of the form $A(R_{\alpha_1,\dots,\alpha_{d-1}})$
where $A$ is an affine map such that $|\det A|\,|R_{\alpha_1,\dots,\alpha_{d-1}}|>\beta\,|D|$. As $\ell\leq L \,n^{1/d}$ for some $L>0$, we obtain \iref{eqn upper bound R alpha} with $C_D=\beta^{-1}\,L^{2+\sum_{i=1}^{d-1}2/ {\alpha_i}}\,C$, the constant $C$ coming from the lemma above.\hfill $\Box$
\newline

We now prove a lower bound based on the same reference domain.

\begin{theorem}
\label{lower bound R alpha}
 Let $D$ be a bounded domain. Assume there exist $\bar x\in \o D$, $0<r_1\leq r_2$, $\alpha_1,\dots,\alpha_{d-1}\in ]0,2]$ and an affine transformation $A$ with $A(0)=\bar x$ such that
 \be
 D\subset A(R_{\alpha_1,\dots,\alpha_{d-1}})\cup \left( \o B(\bar x,r_2)\setminus B(\bar x,r_1)\right).
 \ee
 Then
 \be
 K_{n,D}\geq c_D\, n^{\frac{1}{d}\left(2+\sum_{i=1}^{d-1} 2/{\alpha_i}\right)},
 \ee
 where $c_D$ is a constant depending only on $D$.
\end{theorem}

The proof follows the same path as in Theorem 8.1 and Remark 8.4 of \cite{DP}, but with a radial polynomial centered at $x$ instead of a planar polynomial, that is a univariate polynomial composed with an affine function. This small improvement shows that for a point $x$ and a domain $D$ satisfying the conditions of Theorems \ref{upper bound R alpha} and \ref{lower bound R alpha} with the same $\alpha_i$, the asymptotic behavior of $k_{n,D}(x)$ only depends on $D$ in a neighborhood of $x$.

We first recall Lemma 6.1 from the same article:

\begin{lemma}
 For any $\ell,m\geq 1$ and $y\in [-1,1]$, there exists a univariate polynomial $P_{\ell,m,y}$ of degree at most $\ell$ such that $P_{\ell,m,y}(y) = 1$ and
 \be
|P_{\ell,m,y}(x)|\leq c(m)\,\left(\frac{1+ \ell \sqrt{1-y^2} }{1+ \ell \sqrt{1-y^2} +\ell^2\,|x-y|}\right)^m, \quad x\in [-1,1].
 \ee
 \end{lemma}

Taking $y=-1$ and applying a change of variable $x\mapsto \frac{x+1}2$, we get as an immediate consequence:

\begin{lemma}
\label{bound Pnm}
For any $\ell,m\geq 1$, there exists a univariate polynomial $P_{\ell,m}$ of degree at most $\ell$ such that $P_{\ell,m}(0)=1$ and 
\be
|P_{\ell,m}(x)|\leq c(m) \min\left(1, \frac{1}{\ell^{2m}|x|^m}\right), \quad x\in [0,1].
\ee
\end{lemma}

We also need a bound on the volume of $R_{\alpha_1,\dots,\alpha_{d-1}}$.
\begin{lemma}
\label{calcul volume R alpha}
For all $r>0$, $|R_{\alpha_1,\dots,\alpha_{d-1}}\cap B(0,r)|\leq c\, r^{1+\sum_{i=1}^{d-1}1/\alpha_i}$.
\end{lemma}

\begin{proof}
Given $r\in [0,1]$, $R_{\alpha_1,\dots,\alpha_{d-1}}\cap \{x_d=r\}=[-r^{1/\alpha_1},r^{1/\alpha_1}]\times\dots\times [-r^{1/\alpha_{d-1}},r^{1/\alpha_{d-1}}]\times \{r\}$ has a $(d-1)$-volume equal to $\prod_{i=1}^{d-1} 2\,r^{-\alpha_i}$, so
\be
|R_{\alpha_1,\dots,\alpha_{d-1}}\cap \{0\leq x_d \leq r\}|=\int_0^r \prod_{i=1}^{d-1} 2\,x_d^{1/\alpha_i}\,dx_d=c \,r^{1+\sum_{i=1}^{d-1}1/\alpha_i}.
\ee
As for all $r>0$, $R_{\alpha_1,\dots,\alpha_{d-1}}\cap B(0,r)\subset R_{\alpha_1,\dots,\alpha_{d-1}}\cap \{0\leq x_d \leq \min(1,r)\}$, we obtain the desired result.
\end{proof}

\noindent
{\bf Proof of Theorem \ref{lower bound R alpha}:} 
Take $\ell_0=\lfloor \frac{\ell}{2}\rfloor$, 
$m \geq \frac{1}{2}\sum_{i=1}^{d-1}\frac{1}{\alpha_i}$
 and $r_3\geq r_2$ 
 such that $T(R_{\alpha_1,\dots,\alpha_{d-1}})\subset \o B(\bar x,r_3)$,
  and define the multivariate polynomial
\be
P(x)=P_{\ell_0,m}\left(\frac{|x-\bar x|^2}{r_3^2}\right), \quad x\in \mathbb R^d.
\ee
Then $P$ has degree at most $2\ell_0\leq \ell$ in each variable, $P(\bar x)=1$, and Lemma \ref{bound Pnm} bounds $P$ from above since $D\subset \o B(\bar x, r_3)$. It remains to compute an upper bound of $\|P\|_{L^2(D)}$. For $0<r<r_1$, one has:
\begin{align*}
|D\cap B(\bar x,r)|
&=|T(R_{\alpha_1,\dots,\alpha_{d-1}})\cap B(\bar x,r)|\\
&\leq |\det T|\,|R_{\alpha_1,\dots,\alpha_{d-1}}\cap T^{-1}(B(\bar x,r))|\\
&\leq  |\det T|\,|R_{\alpha_1,\dots,\alpha_{d-1}}\cap B(0,r\,\lambda_{\max}(T^{-1}))|\\
& \leq c' r^{1+\sum_{i=1}^{d-1}1/\alpha_i},
\end{align*}
where in the last line we used Lemma \ref{calcul volume R alpha}, and with $c'=\frac{c\,|\det T|}{\lambda_{\min}(T)^{1+\sum_{i=1}^{d-1}1/\alpha_i}}$. Therefore, one can compute
\begin{align*}
\|P\|_{L^2(D)}^2&\leq \left\|c(m) \min\left(1, \frac{1}{\ell^{2m}|x|^m}\right)\right\|_{L^2(D)}^2\\
&\leq c(m)^2\left(\int_{D\cap B(\bar x,\ell^{-2})}\!\!\! dx
+\int_{B(\bar x,r_3)}\frac {dx}{\ell^{4m}r_1^{2m}} 
+\int_{D\cap B(\bar x,r_1)\setminus B(\bar x,\ell^{-2})}\!\left(\frac{1}{|x|^{2m}}-\frac{1}{r_1^{2m}}\right)\frac{dx}{\ell^{4m}}\right)\\
&= c(m)^2\left( \left|D\cap B\left(\bar x,\ell^{-2}\right)\right|
+\frac{|B(\bar x,r_3)|}{\ell^{4m}r_1^{2m}}
+\int_{\ell^{-2}}^{r_1} \frac{2m}{\ell^{4m}r^{2m+1}}\,|D\cap B(\bar x,r)|\,dr\right)\\
&\leq c(m)^2\left(c' \ell^{-2-\sum_{i=1}^{d-1} 2/\alpha_i}
+\frac{|B(\bar x,r_3)|}{\ell^{4m}r_1^{2m}}
+c'\,\frac{2m}{\ell^{4m}}\int_{\ell^{-2}}^{r_1}r^{\sum_{i=1}^{d-1}1/\alpha_i-2m}dr\right)\\
&\leq c''\max\left(\ell^{-2-\sum_{i=1}^{d-1} 2/\alpha_i}, \ell^{-4m}, \ell^{-4m-2(\sum_{i=1}^{d-1}1/\alpha_i-2m+1)}\right)\\
&=c''\ell^{-2-\sum_{i=1}^{d-1} 2/\alpha_i},
\end{align*}
and conclude that $K_{n,D}\geq k_{n,D}(\bar x)\geq \frac{|P(\bar x)|^2}{\|P\|_{L^2(D)}^2}\geq c_D\, \ell^{2+\sum_{i=1}^{d-1} 2/\alpha_i}$, with $c_D=1/c''$.\hfill $\Box$

\begin{remark}
These theorems include the case of smooth domains : indeed, taking $\alpha_1=\dots=\alpha_{d-1}=2$ and $e_d=(0,\dots,0,1)$, one has
\be
B\left(\frac{1}{2}e_d,\frac{1}{2}\right)\subset R_{2,\dots,2}\subset \left\{x\in \mathbb [-1,1]^d, \ \frac1{d-1}\sum_{i=1}^{d-1} |x_i|^2\leq x_d\right\}\subset B\left((d-1)e_d,(d-1)\right),
\ee
so one can recover the results \ref{upper bound smooth} and \ref{lower bound smooth}, without explicit constants.
Similarly, Lipschitz boundaries correspond to the particular values $\alpha_1=\dots=\alpha_{d-1}=1$.
\end{remark}

\begin{example}
It becomes useful to take distinct values for the $\alpha_i$ in the case of domains with edges but no corners. For instance, consider $D=\frac{\sqrt{3}}{2}e_d+B\left(\frac{1}{2}e_1,1\right)\cap B(-\frac{1}{2}e_1,1)$. Then $0\in A(R_{1,2\dots,2})\subset \o D\subset B(R_{1,2\dots,2})$, where $A$ and $B$ are the linear maps defined by
\be
A(x_1,\dots,x_d)=\left(\frac{1}{4}\,x_1,\frac{1}{2\sqrt{d-2}}\,x_2,\dots,\frac{1}{2\sqrt{d-2}}\,x_{d-1} ,\frac{\sqrt 3}{2}\,x_d\right)
\ee
and
\be
B(x_1,\dots,x_d)=\left( 3\,x_1,\frac{1}{\sqrt{3}}\,x_2,\dots,\frac{1}{\sqrt{3}}\,x_{d-1},\sqrt 3\, x_d\right).
\ee
Thus $K_{n,D}\geq k_{n,D}(0)\sim \ell^{d+2}$. Moreover, for all $x\in D$ there exists an affine transformation $T$ such that $\det T\geq 2^{-d}$,  $T(D)\subset D$ and $T(0)=x$, so $K_{n,D}\sim \ell^{d+2}$.
\end{example}

\begin{remark}
It is easily seen that for domains having a cusp that points outside,
the value of $K_n$ may grow as fast as any polynomial, depending on the order of cuspitality.
For instance, given $\alpha\in ]0,2]$, according to Theorems \ref{upper bound R alpha} and \ref{lower bound R alpha}, one has 
\be
k_{n,R_{\alpha,\dots,\alpha}}(0)\sim \ell^{2+\frac{2}{\alpha}(d-1)},
\ee
so that $K_{n,R_{\alpha,\dots,\alpha}}\geq c\,\ell^{2+\frac{2}{\alpha}(d-1)}$.
\end{remark}

\section{Numerical illustration}

In this section we give numerical illustrations of the offline 
and online sampling strategies in the particular case of algebraic 
polynomials and for different domains. As in the previous, we consider 
spaces polynomials of fixed total degree $V_n=\P_\ell$
as defined by \iref{total}. 

The three considered domains are 
\begin{enumerate}
\item
$D:=\{x_1^2+x_2^2\leq \frac 2\pi\}$, the ball of area $2$.
\item
$D:=\{ -1\leq x_1\leq 1, \; |x_1|-1\leq x_2\leq |x_1|\}$, a polygon with 
a reintrant corner at $(0,0)$.
\item
$D:=\{ -1\leq x_1\leq 1, \; \sqrt{|x_1|}-1\leq x_2\leq \sqrt{|x_1|}\}$, a domain
with a reintrant cusp at $(0,0)$.
\end{enumerate}
The measure $\mu$ for the error metric $L^2(D,\mu)$ is 
the uniform probability measure on the considered domain.
In all three cases, the domain $D$ is embedded in the unit cube $Q=[-1,1]^2$,
and described by algebraic inequalities. Thus, sampling according to $\mu$
is readily performed by uniform sampling on $Q$ which is done separately 
on the two coordinates followed by rejection when $x\notin D$.

\begin{figure}[ht]
\begin{center}
\includegraphics[width=16cm]{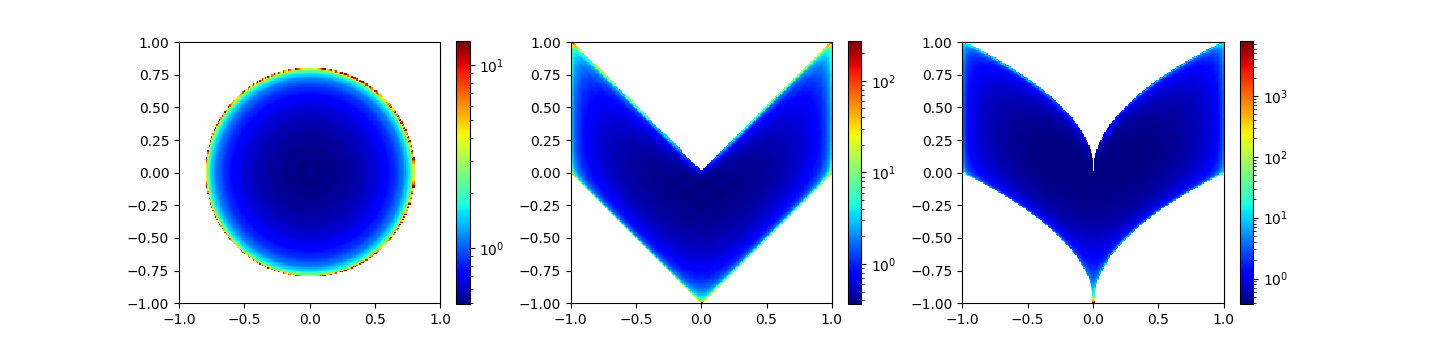}
\caption{\label{figdom} The three domains (disc, polygon, cusp) and the function $k_n/n$ for $n=231$.}
\end{center}
\end{figure}

The above three domains are instances of smooth, Lipschitz and cuspital
domains, respectively. They are meant to illustrate how
the smoothness of the boundary affects the amount of sample needed
in the offline state, as rigorously analyzed in the previous section.
On these particular domains, we are actually able to exactly 
integrate polynomials, and therefore in principle to compute the exact orthogonal
polynomials $L_j$ up to round-off error due to the orthogonalization procedure.
In our numerical tests, the considered total degrees are $\ell=0,1,\dots,20$, therefore $n=n(\ell)=1,3,6,\dots,231$. The intermediate values of $n$
between $n(\ell)$ and $n(\ell+1)$ are treated by complementing the 
space $V_n$ with the monomials $x_1^{\alpha_1}x_2^{\alpha_2}$ for $\alpha_1+\alpha_2=\ell+1$ in the order $\alpha_2=0,\dots,\ell+1$.
For such values, we could compute the $L_j$ using Cholesky factorization with quadruple precision,
and check that $|\<L_j,L_k\>-\delta_{j,k}|\leq 10^{-16}$, that is, orthonormality holds up to double precision.

We may thus compute for each value of $n$ the exact inverse Christoffel 
function $k_n$ and optimal measure $\sigma^*=\frac {k_n}{n}\, \mu$. 
Figure \ref{figdom} displays the three domains and the value of $k_n/n$ for the maximal value $n=231$
which, as explained by the results in \S 5, grows near to the boundary,
faster at the exiting corners (and even faster at exiting cusps), and slower
in smooth regions or at reintrant singularities.

This exact computation 
allows us to compare the optimal sampling strategy based on $\sigma^*$
and the more realistic strategy based on $\wt \sigma$ which is computed
from the approximate inverse Christoffel function $\wt k_n$ derived in the offline stage.
We next show that both strategies perform similarly well in terms of instance optimality 
at near-optimal sampling budget. We stress however that for more general domains where exact integration of polynomials is not feasible, only the second strategy based on $\wt k_n$ is viable.

\subsection{Sample complexity of the offline stage}

We first illustate the sample complexity $M$ in the offline stage.
As discussed \S 3.2, a sufficient condition to ensure the 
framing \iref{compk} between $k_n$ and $\wt k_n$
is the matrix framing property \iref{matframe} which expresses
the fact that the condition number of $G_M$ satisfies the bound
\be
\kappa(G_M)\leq c=\frac{c_2}{c_1}.
\ee
For the constants $c_1=\frac 2 3$
and $c_2=2$, this occurs with high probability when $M$ is larger than 
$K_n$, or a known upper bound $B(n)$, multiplied by logarithmic factors, as expressed by \iref{condM1}.

\begin{figure}[ht]
\begin{center}
\includegraphics[width=16cm]{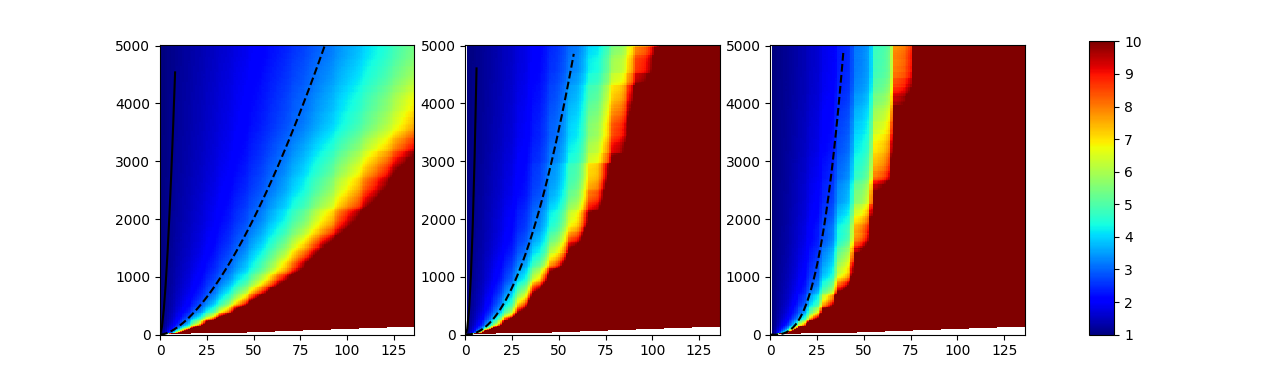}
\caption{\label{figcondM} Conditioning of the matrix $G_M$
for the disc (left), polygon (center) and cusp (right) domains,
with theoretical value of $M_{\rm suf}(n)$ (full curve) and adjusted
value $M_{\rm adj}(n)$ (dashed curve). The x-coordinate stands for $n$ and the y-coordinate for $M$; moreover, the plotted values are saturated at 10 since we are only interested in small condition numbers.}
\end{center}
\end{figure}

Figure \ref{figcondM} displays the condition number $\kappa(G_M)$,
averaged over $100$ realizations of the offline sample $\{z^1,\dots,z^M\}$,
as a function of $n$ and $M\geq n$, 
for the three considered domains. We observe a transition region
that illustrates the minimal offline sampling budget $M_{\min}(n)$
that should be practically invested in order for $G_M$ to be well
conditioned, for exemple such that $\kappa(G_M)\leq 3$.

We also draw in full line the value of the sufficient value 
\be
M_{\rm suf}(n):=\gamma\, B(n)\ln(2n/\eps),
\ee 
for $\eps=10^{-2}$ where $B(n)$ are the upper
bounds for $K_n$ derived from the
theoretical analysis of \S 5.
These upper bounds are $3n^{3/2}$ 
for the disc in view of \iref{ellipsoid} and $2n^{2}$ for 
the polygonal domain by application of Proposition \ref{lemmaref}
with $\beta=\frac 1 2$, since $D$ is the union of two parallelograms
of equal size.  While the sampling budget $m=M_{\rm suf}(n)$
guarantees that $\kappa(G_M)\leq 3$ with high probability - here 0.99 -
the plots reveal that this budget is by far an over-estimation of $M_{\min}(n)$.

We draw in dashed line the adjusted values 
$M_{\rm adj}(n)=C_{\rm adj}\,M_{\rm suf}(n)$
where the multiplicative constant is picked as small
as possible with the constraint of still fitting
requirement $\kappa(G_M)\leq 3$, thus better fitting 
the minimal budget  $M_{\min}(n)$. We find that 
constant $C_{\rm adj}$ is approximately $\frac 1 {45}$
for the disc and $\frac 1 {120}$ for the polygon. 
It is even smaller for the cusp domain, for which Theorem \ref{upper bound R alpha}
with $\alpha_1=\frac 1 2$ yields an upper bound
of the form $B(n)=Cn^{3}$ with a constant $C$ that can be numerically
estimated but turns out to be very pessimistic.

In summary, the offline sampling budget $M_{\rm suf}(n)$ suggested by 
the theoretical analysis is always pessimistic by a large multiplicative 
constant. Let us remind that the value $M_{\rm min}(n)$ is typically not accessible
to us since $G_M$ and its condition number cannot be exactly evaluated
for more general domains $D$. 

\begin{figure}[ht]
\begin{center}
\includegraphics[width=16cm]{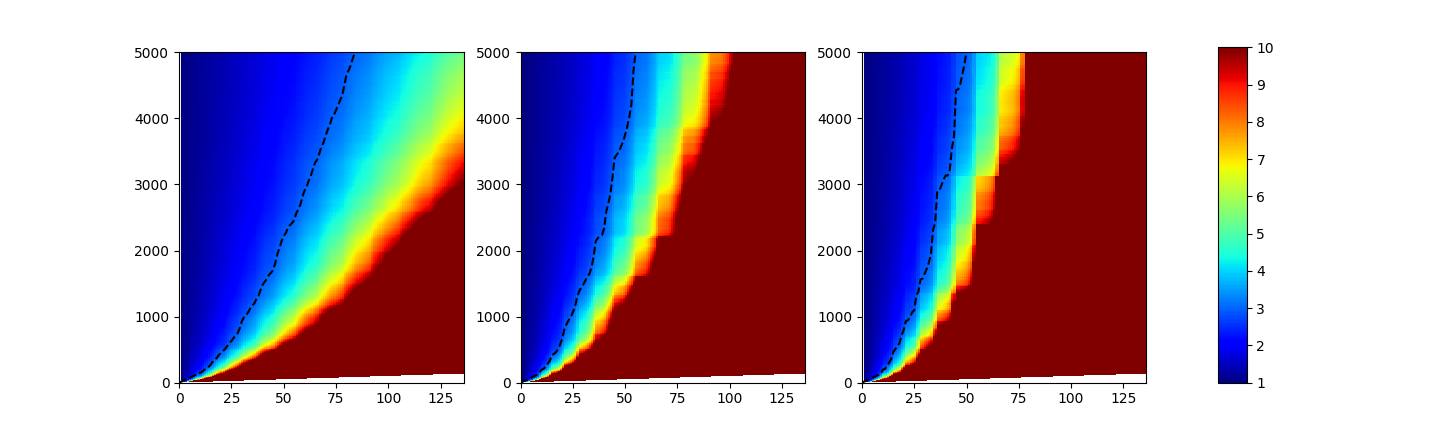}
\caption{\label{figcondT} Conditioning  of the matrix $T$
for the disc (left), polygon (center) and cusp (right) domains,
and value of $M_{\rm emp}(n)$ (dashed curve).}
\end{center}
\end{figure}

This state of affair justifies the use of the empirical method outlined in \S 3.3 for selecting a good value of $M$.
Recall that this approach consists in raising $M$ until the conditioning of the 
computable matrix $T$ becomes less than some prescribed value, for example $\kappa(T)\leq 3$. Figure~\ref{figcondT} displays the 
conditioning $\kappa(T)$ again
averaged over $100$ realizations of the offline sample, as well
as the curve showing the empirical value $M_{\rm emp}(n)$ which 
corresponds to the smallest value of $M$ such that $\kappa(T)\leq 3$.
It reveals the relevance of the empirical approach: due to the
very good fit between $\kappa(T)$ and $\kappa(G_M)$,
the value $M_{\rm emp}(n)$ appears as a much sharper estimate
for $M_{\rm min}(n)$ than $M_{\rm suf}(n)$.

\subsection{Sample complexity of the online stage}

We next study the sample complexity $m$ of the online stage
through the conditioning of the matrix $G=(\<L_j,L_k\>_m)_{j,k=1,\dots,n}$,
where $\<\cdot,\cdot\>_m$ is the inner product associated to the 
discrete norm
\be
\|v\|_m^2:=\frac 1 m\sum_{i=1}^m w(x^i)\,|v(x^i)|^2.
\ee

For the sampling measure $\sigma$ and weight $w$, we both consider:
\begin{enumerate}
\item[(i)]
The optimal sampling measure
$d\sigma^*:=\frac {k_n}{n}\, d\mu$ and weight $w^*=\frac n{k_n}$,
which, for these particular domains, can be exactly computed from the $L_j$,
but are not accessible for more general domains.
\item[(ii)]
The empirical sampling measure
$d\wt\sigma:=\frac {\wt k_n}{n} d\mu$ and weight $\wt w=\frac n{\wt k_n}$
where $\wt k_n$ has been obtained from the offline stage,
using the previously described empirical choice of $M$. 
\end{enumerate}

\begin{figure}[ht]
\begin{center}
\includegraphics[width=16cm]{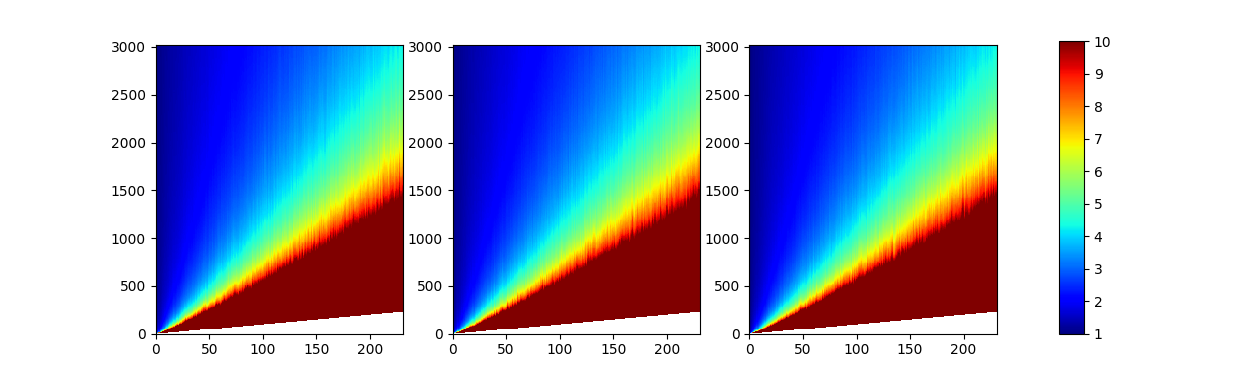}
\includegraphics[width=16cm]{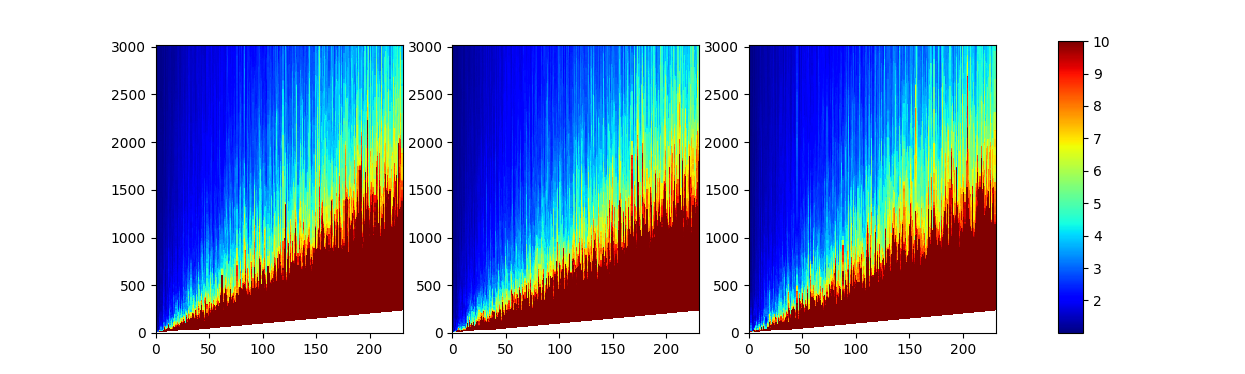}
\caption{\label{weightedtkn} Conditioning of $G=\<L_j,L_k\>_m$ depending on $m$ and $n$
for the disc (left), polygon (center) and cusp (right) domains, using $k_n$ (up) or the estimated $\wt k_n$ (down)}
\end{center}
\end{figure}

Figure \ref{weightedtkn}
displays the condition number $\kappa(G)$, as a function of $m$ and $n$,
for both choices and the three domains. 
In order to illustrate the fluctuations of $\kappa(G)$, we 
display an averaging over $100$ realizations when using $k_n$, 
and one single realization when using $\wt k_n$. 
While the behaviour for a single realization is more chaotic,
we find that in both case, as expected, the online sampling budget $m(n)$ 
which ensures that $G$ is well conditioned, for example $\kappa(G)\leq 3$, 
grows linearly with $n$ (up to logarithmic factors), now independently of the
domain shape. 

\subsection{Instance and budget optimality}

In order to illustrate the achievement of our initial goal of instance 
and budget optimality, we consider the approximation in 
a polynomial space $V_n=\P_\ell$ of a function 
$u$ that consists of a polynomial part $u_n\in V_n$
and a residual part $u_n^\perp \in V_n^\perp$
that are both explicitly given in terms of their expansions
\be
u_n=\sum_{j=1}^n c_j L_j,
\ee
and
\be
u_n^\perp=\sum_{j\geq n+1} c_j L_j.
\ee
For numerical testing, we take only finitely many non-zero $c_j$ in this second
expansion and adjust them so that $\sum_{j\geq n+1} |c_j|^2=10^{-4}$.
Thus, the best approximation error has value
\be
e_n(u)=\|u-u_n\|=\|u_n^\perp\|=10^{-2}.
\ee
We study the mean-square error $\E(\|u-P^m_nu\|^2)$ as a function of $m$
 and compare the different sampling strategies through their ability to reach
this ideal benchmark.

\begin{figure}[ht]
\begin{center}
\includegraphics[width=17cm,height=4cm]{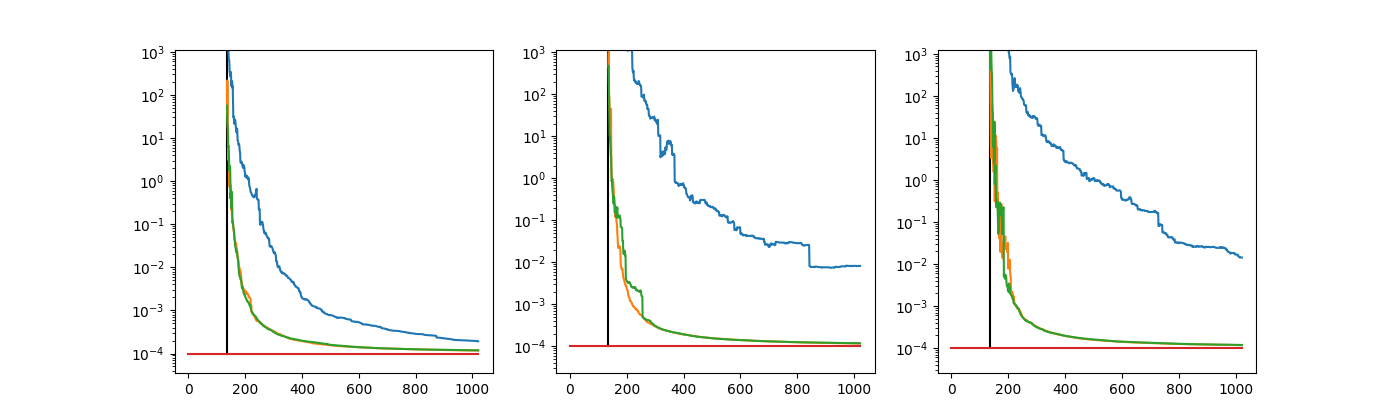}
\caption{\label{error_square}
Mean-square reconstruction error for the disc (left), polygon (center) and cusp (right) domains, with total polynomial degree $\ell=15$,
and sampling measures $\mu$ (blue), $\sigma^*$ (orange), $\wt \sigma$ (green).
Horizontal red line: best approximation error $e_n(u)^2=10^{-4}$. Vertical black line: polynomial dimension $n=136$.}
\end{center}
\end{figure}

Figure \ref{error_square} displays the error curves (obtained by 
averaging $\|u-P^m_nu\|^2$ over $100$ realizations)
for the three domains and polynomial degree $\ell=15$ that corresponds to the dimension $n=136$. For all domains, we observe that
the best approximation error is attained up to multiplicative factor $2$
with a sampling budget $m$
that is thrice larger than $n$, when using either the
optimal sampling measure $\sigma^*$ based on $k_n$
or the measure $\wt \sigma$ based on $\wt k_n$
obtained in the offline stage. 
This does not occur when sampling
according to the uniform measure $\mu$: the error remains
orders of magnitude above the best approximation error 
and this effect is even more pronounced
as the domain becomes singular. This reflects the fact that with the uniform 
sampling, the budget $m$ needs
to be larger than $K_n$ which has faster growth with $n$ for singular domains.



\begin{thebibliography}{77}


\bibitem{AH} B. Adcock and D. Huybrechs, {\it Approximating smooth, multivariate functions on irregular domains}, Forum of Mathematics, Sigma, vol.8, Cambridge University Press, 2020.

\bibitem{ABC} B. Arras, M. Bachmayr and A. Cohen,  {\it Sequential sampling for optimal weighted least squares
approximations in hierarchical spaces}, SIAM Journal on Mathematics of Data Science, 1 (2019), pp. 189-207.

\bibitem{Chkifa} A. Chkifa, {\it On the Lebesgue constant of Leja sequences for the complex unit disk and of their real projection}, Journal of Approximation Theory, 166 (2013), pp. 176-200.

\bibitem{CCMNT} A. Chkifa, A. Cohen, G. Migliorati, F. Nobile and R. Tempone, {\it Discrete 
least squares polynomial approximation with random evaluations - Application to 
parametric and stochastic elliptic PDEs}, ESAIM: Mathematical Modelling and Numerical Analysis, 49 (2015), pp. 815-837.

\bibitem{CD} A. Cohen and R. DeVore, {\it High dimensional approximation of parametric PDEs}, Acta Numerica (2015).

\bibitem{CDL} A. Cohen, M. Davenport and D. Leviatan, {\it On the stability and accuracy of least squares approximations},
Foundations of computational mathematics, 13 (2013), pp. 819-834.

\bibitem{CM1} A. Cohen and G. Migliorati, 
\emph{Optimal weighted least squares methods,} SMAI Journal of Computational Mathematics \textbf{3}, 181--203, 2017.

\bibitem{DL} R. A. DeVore and G. Lorentz, {\it Constructive approximation}, vol. 303, Springer Science \& Business Media, 1993.

\bibitem{DH} A. Doostan and J. Hampton, {\it Coherence motivated sampling and convergence analysis of least squares polynomial Chaos regression}, Computer Methods in Applied Mechanics and Engineering, 290 (2015), pp. 73-97.

\bibitem{DP} Z. Ditzian and A. Prymak, {\it On Nikol'skii Inequalities for Domains in $\mathbb R^d$}, Constructive Approximation, 44 (2016), pp. 23-51.

\bibitem{HNP} C. Haberstich, A. Nouy, and G. Perrin, 
{\it Boosted optimal weighted least-squares}, arXiv:1912.07075, 2019.

\bibitem {Kro} A. Kro\'o, {\it Christoffel functions on convex and starlike domains in $\mathbb R^d$}, Journal of Mathe- matical Analysis and Applications, 421 (2015), pp. 718-729.

\bibitem{JNZ} J.D. Jakeman, A. Narayan, and T. Zhou, {\it A Christoffel function weighted least squares algorithm for collocation approximations}, Mathematics of Computation, 86 (2017), pp. 1913-1947.

\bibitem{Magic} Y. Maday, N.C. Nguyen, A.T. Patera, 
and G.S.H. Pau, {\it A general multipurpose interpolation procedure: the magic points},
Communications on Pure \& Applied Analysis, 8 (2009), p. 383.

\bibitem{Mig2} G. Migliorati, {\it Adaptive approximation by optimal weighted least-squares methods}, SIAM Journal on Numerical Analysis, 7 (2019), pp. 2217-2245.

\bibitem{Mig} G. Migliorati, {\it Multivariate approximation of functions on irregular domains by weighted least-squares methods}, IMA journal of numerical analysis (2020),
https://doi.org/10.1093/imanum/draa023.

\bibitem{Pry} A. Prymak and O. Usoltseva, {\it Christoffel functions on planar domains with piecewise smooth boundary}, Acta Mathematica Hungarica, 158 (2019), pp. 216-234.

\bibitem{Tr} J. Tropp, {\it User-Friendly tail bounds for sums of random matrices},
Foundations of computational mathematics, 12 (2012), pp. 389-434.

\bibitem{Xu} Y. Xu, {\it Asymptotics for orthogonal polynomials and christoffel functions on a ball}, Methods and Applications of Analysis, 3 (1996), pp. 257-272.

\end{thebibliography}
\end{document}